\documentclass[11pt]{article}
\usepackage{authblk}
\usepackage{bbold}
\usepackage{amsmath,amsfonts,amssymb,amsthm}

\usepackage{tikz-cd}

\theoremstyle{plain}

\newtheorem{theorem}{Theorem}[section]
\newtheorem{proposition}[theorem]{Proposition}
\newtheorem{lemma}[theorem]{Lemma}
 
\newtheorem{definition}[theorem]{Definition}
 
\newenvironment{remark}
  {\par\noindent\textbf{Remark. }\normalfont}
  {\hfill$\blacksquare$\par}

\newcommand{\HH}{\mathcal{H}}

\newcommand{\ip}[2]{\left\langle #1,#2 \right\rangle}

\newcommand{\RR}{\mathbb{R}}
\newcommand{\NN}{\mathbb{N}}

\newcommand{\cZ}{{\cal Z}}

\newcommand{\cK}{{\cal K}}
\newcommand{\cF}{{\cal F}}

\newcommand{\cP}{{\cal P}}

\newcommand{\mC}{{\mathbb C}}

\newcommand{\mR}{{\mathbb R}}
\usepackage{algorithm}
\usepackage{algorithmic}
\usepackage[algo2e,linesnumbered,ruled]{algorithm2e}

\newcommand{\mU}{{\mathbb U}}

\newcommand{\bff}{{\mathbf f}}

\newcommand{\bC}{{\boldsymbol{\mathcal C}}}

\usepackage{fullpage}
\usepackage{hyperref}
\usepackage{graphicx}
 \usepackage{cite}
\usepackage{url}
\def\[{\begin{equation}}
\def\]{\end{equation}}
\usepackage[utf8]{inputenc}
\usepackage[margin=1.5in]{geometry}
\usepackage{algorithmic} 

\title{Kernel Methods for Some Transport Equations with Application to Learning Kernels for the Approximation of Koopman Eigenfunctions: \\ A Unified Approach via Variational Methods, Green’s Functions and the Method of Characteristics
}
 \author[1,2]{Boumediene Hamzi}
\author[1]{Houman Owhadi}
\author[3]{Umesh Vaidya}

\affil[1]{Department of Computing and Mathematical Sciences, Caltech, Pasadena, CA, USA}
\affil[2]{The Alan Turing Institute, London, UK}

\affil[3]{Department of Mechanical Engineering, Clemson University, Clemson, SC, USA}

\date{}

\begin{document}
\maketitle

\begin{abstract}
We present a unified theoretical and computational framework for constructing reproducing kernels tailored to transport equations and adapted to Koopman eigenfunctions of nonlinear dynamical systems. These eigenfunctions satisfy a transport-type partial differential equation (PDE) that we invert using three analytically grounded methods: (i) A Lions-type variational principle in a reproducing kernel Hilbert space (RKHS), (ii) convolution with a Green’s function,  and (iii) a resolvent operator constructed via Laplace transforms along characteristic flows. We prove that these three constructions yield identical kernels under mild smoothness and causality assumptions. We further show that the associated kernel eigenfunctions (Mercer modes) converge in \( L^2 \) to true Koopman eigenfunctions when the latter lie in the RKHS. Our approach is numerically realized through a mesh-free, convex optimization framework, enhanced with boundary regularization to handle eigenfunction blow-up. A multiple-kernel learning (MKL) scheme selects kernels automatically via residual minimization. Finally, we demonstrate that the same framework applies verbatim to a broader class of linear transport PDEs, including the advection, continuity, and Liouville equations.

The unification of variational principles, Green’s functions, and the method of characteristics enables the development of novel schemes for approximating eigenfunctions of transport equations, including those of the Koopman operator, and introduces a data-driven approach for learning kernels tailored to these approximations. Numerical experiments confirm the practical utility and robustness of the method.
\end{abstract}

\tableofcontents

\section{Introduction}

The theory of reproducing kernel Hilbert spaces (RKHSs) has long served as a powerful framework for both the analysis and numerical approximation of solutions to partial differential equations (PDEs). In pioneering work, Jacques‑Louis Lions \cite{lions1986noyau, lions1988remarks} demonstrated that one can “invert” an elliptic operator via a variational formulation in a suitably chosen Hilbert space-typically a Sobolev space where point evaluations are continuous-to obtain a reproducing kernel. Specifically, Lions showed that every function in such a space can be recovered via the inner product with a unique kernel, which, under favorable conditions, admits a spectral expansion of the form:
\[
K(x,y)= \sum_{j=0}^\infty \psi_j(x)\,\psi_j(y).
\]

Building on these ideas, Engliš, Lukkassen, Peetre, and Persson \cite{englis2004formula} generalized this variational framework to a broader class of elliptic operators and boundary value problems. Their derivations often expressed the reproducing kernel in terms of classical Green’s or Poisson kernels, thereby highlighting how the kernel encapsulates both the geometric structure of the domain and the analytic properties of the differential operator. In parallel, Auchmuty \cite{auchmuty2009reproducing} developed an alternative approach by constructing Hilbert spaces of harmonic functions via spectral (Steklov) expansions. His method characterizes the function space in terms of boundary trace spaces and represents the reproducing kernel explicitly as an expansion in the Steklov eigenfunctions. This not only recovers the reproducing property but also provides a concrete formula for the kernel, reinforcing the variational principles initially introduced by Lions.

On another side, time series data, ubiquitous across scientific disciplines, have motivated a broad range of forecasting methods grounded in statistical and machine learning approaches \cite{kantz97,CASDAGLI1989, yk1, yk2, yk3, yk4, survey_kf_ann,Sindy,jaideep1,nielsen2019practical,abarbanel2012analysis, pillonetto2011new,wang2011predicting,brunton2016discovering,lusch2018deep,callaham2021learning,kaptanoglu2021physicsconstrained,kutz2022parsimony}. Dynamical systems theory provides tools to understand the governing rules underlying such data. 

Originating with the foundational work of Koopman \cite{koopman1932dynamical}, Koopman operator theory offers a linear lens through which to study nonlinear dynamics. More recently, the study of dynamical systems has increasingly focused on the Koopman operator-a linear, infinite-dimensional operator that governs the evolution of observables along trajectories of nonlinear dynamical systems \cite{mezic2005spectral}. A central goal in Koopman analysis is to extract the eigenfunctions of this operator, as they reveal the intrinsic modal structure of the system. 

The techniques pioneered by Lions and refined by Engli\v{s}
 et al.\ and Auchmuty provide a natural bridge from classical variational principles to the spectral theory of reproducing kernels. Just as harmonic functions can be recovered via kernel-based variational formulations, we construct reproducing kernels whose Mercer eigenfunctions approximate Koopman eigenfunctions. Our approach not only unifies these perspectives but also yields a practical numerical scheme: a strictly convex optimization problem in a reproducing kernel Hilbert space whose minimizers capture the desired eigenfunctions with stability and regularization. As a linear operator, the Koopman operator's spectral decomposition encodes rich information crucial for analysis and synthesis. Principal eigenfunctions, for instance, reveal the geometry of the state space; their joint zero-level sets characterize stable and unstable manifolds of equilibrium points \cite{mezic2005spectral}. These eigenfunctions are associated with the limit set or stationary dynamics such as an equilibrium point or limit cycle of the dynamical systems. For example, with the equilibrium dynamics the principal eigenfunctions will have eigenvalues same as that of the linearization of the system at that equilibrium point.  In \cite{mauroy2016global,umathe2023spectral,matavalam2024data}, these eigenfunctions are used to identify the stability boundary and domain of attraction of an equilibrium. Additionally, they play a key role in optimal control synthesis, as shown in \cite{vaidya2022spectral,VaidyaKHJ2024}. A path-integral formulation for computing Koopman eigenfunctions was introduced in \cite{deka2023path}. However, challenges remain: the Koopman operator often has both discrete and continuous spectra \cite{mezic2005spectral}, and approximating an infinite-dimensional operator via finite-dimensional projections can lead to "spectral pollution" \cite{colbrook2020foundations}.

In this paper, we propose a numerical method for computing principal Koopman eigenfunctions directly, sidestepping the need to construct the Koopman operator explicitly. In particular, we propose an approach for computing the principal eigenfunctions of the Koopman operator associated with the equilibrium dynamics. Our method builds on a decomposition principle introduced in \cite{deka2023path}, wherein the principal eigenfunctions associated with an equilibrium point are partially determined by the linearization of the vector field at the equilibrium. 
Our work is also inspired by the broader literature on data-driven Koopman learning, such as extended dynamic mode decomposition (EDMD) \cite{williamsEDMD}, which approximates the Koopman operator using actions on a space of pre-selected basis functions. From this point of view, our work proposes a method to learn the basis functions from data inspired by Lions's variational approach to contstruct kernels.  Recent advances have combined kernel methods with Koopman operator theory \cite{klus2020data,peherstorfer2020koopman,klus2015kernel,DAS2020573,hou2023sparse,ishikawa2024koopmanoperatorsintrinsicobservables}, leveraging the RKHS framework \cite{CuckerandSmale} to improve regularization, convergence, and interpretability \cite{chen2021solving,  houman_cgc}. Kernel methods now play a central role in dynamical systems analysis \cite{lee2025kernel, hamzi2018note, bouvrie2010balanced, bouvrie2012empirical, hamzi2019kernel,  yk1, bhcm11, bhcm1, lyap_bh, BHPhysicaD, hamzi2019kernel, bh2020b, klus2020data, ALEXANDER2020132520, bhks, bh12, bh17, hb17, mmd_kernels_bh, bh_kfs_p5,bh_kfs_p6,hou2024propagating,cole_hopf_poincare,kernel_sos}, as well as in surrogate modeling \cite{santinhaasdonk19}.

 
We use the fact that the principal eigenfunctions have associated eigenvalues same as that of the linearization of the system at an equilibrium point to reduce the spectrum computation problem to the solution of linear advection-type PDE. 
To that end, we adapt classical ideas from Lions to the setting of transport-type (or advection) eigenvalue problems:
\[
f(x)\cdot \nabla \varphi(x)=\lambda\,\varphi(x), \quad \text{or equivalently} \quad L\varphi(x):=f(x)\cdot \nabla \varphi(x)-\lambda\, \varphi(x)=0.
\]
Our goal is to construct a reproducing kernel \(K(x,y)\) for the space of solutions in a unified and numerically tractable way. We pursue two  direct approaches:  
\begin{enumerate}
    \item The Green’s function method, which constructs \( K(x, y) \) by symmetrizing a causal Green’s function \( G(x, \xi) \) solving the PDE \( L_x G(x, \xi) = \delta(x - \xi) \). 
    \item The method of characteristics. This approach leverages the structure of the flow generated by the vector field \( f(x) \) to solve the PDE by integrating along characteristic curves. By expressing the solution operator as a semigroup acting via pullback along the flow, and taking its Laplace transform, one obtains a resolvent kernel that serves as a reproducing kernel for the solution space. This construction is particularly natural when the flow map is explicitly known or can be efficiently simulated, and is deeply connected to Koopman operator theory. The resulting kernel is not only consistent with Lions's variational solution and the Green’s function convolution, but also provides a geometric and dynamically-informed representation of the solution space.
\end{enumerate}

Under suitable assumptions, \emph{these constructions-together with the variational framework-yield equivalent reproducing kernels}. That is, the kernel obtained via Green’s function coincides with that derived from the variational or characteristic methods. This also allows to propose a new approach to learn kernels that approximate Koopman's eigenfunctions from data. This is illustrated in the  following diagram
 \[
\begin{tikzcd}[row sep=huge, column sep=huge]
\text{Green's Function} \arrow[dr] & \\
\text{Lions's Variational Approach} \arrow[r] \arrow[u, dashed] \arrow[d, dashed] & \boxed{\text{Reproducing Kernel}} \\
\text{Method of Characteristics} \arrow[ur] &
\end{tikzcd}
\]

Finally, while our primary interest lies in constructing Koopman eigenfunctions, our methods are general enough to apply to a range of transport-type equations, including the linear advection equation, continuity equation, Liouville equation, and decay transport equations. 

This unification of classical reproducing kernel theory with modern dynamical systems  analysis,  via variational methods, provides a rigorous and flexible framework for  the data-driven study of nonlinear dynamics. It enables the development of novel schemes for approximating eigenfunctions of transport equations, including those of the Koopman operator, and introduces a data-driven approach for learning kernels tailored to these approximations.

\paragraph{Our contributions.}
\begin{itemize}
  \item \textbf{Kernel unification theorem.} We prove that the reproducing kernel obtained via
        \emph{(i)} Lions’s variational inversion, \emph{(ii)} causal Green’s function symmetrization,
        and \emph{(iii)} the Laplace resolvent of the characteristic flow are mathematically identical
        under mild smoothness and causality assumptions.

  \item \textbf{Spectral convergence of kernel modes.} We prove that the Mercer eigenfunctions of the constructed kernels converge (in \(L^2\)) to Koopman eigenfunctions when the latter lie in the RKHS. This bridges kernel-based and spectral methods and provides theoretical support for data-driven kernel approximations.

  \item \textbf{Variational solver for principal Koopman eigenfunctions.} We recast the Koopman
        eigenvalue PDE as a convex optimization problem in an RKHS, yielding a mesh-free,
        regularized algorithm that remains well-posed even when eigenfunctions exhibit boundary
        blow-up.

  \item \textbf{Handling blowing eigenfunctions.} For systems in which Koopman eigenfunctions diverge
        near the boundary (e.g.\ the 1D cubic example), we introduce boundary–trace and boundary-layer
        penalties that confine the optimisation to an interior subdomain and guarantee existence,
        uniqueness, and numerical stability.
\item \textbf{Numerical strategies for singular eigenfunctions.} We implement boundary trace and boundary layer penalties to stably recover Koopman eigenfunctions that exhibit boundary divergence. These techniques enable consistent numerical recovery even when eigenfunctions are not square-integrable on the full domain.

  \item \textbf{Dynamics-informed kernels.} All kernels in our framework are derived by solving the transport-type PDE \( (f \cdot \nabla - \lambda)\phi = 0 \) through three analytically grounded constructions-Green’s functions, variational inversion, and flow-based resolvents-ensuring that the RKHS structure encodes transport geometry and Koopman spectral features.

\item \textbf{Path-integral kernel construction.} We propose a kernel construction method inspired by path-integral formulations to compute Koopman eigenfunctions.

  \item \textbf{Unsupervised kernel learning.} A multiple-kernel-learning (MKL) scheme tunes kernel weights
        by minimizing the Koopman residual itself, automatically balancing expressivity and smoothness
        without manual parameter tuning.

  \item \textbf{Extensibility and validation.} Our framework applies verbatim to other linear transport equations-including advection, continuity, Liouville, and decay models-demonstrating its broad scope and practical utility. This extensibility is supported by analytical equivalence and illustrated in Section~5.

\item \textbf{Koopman-spectral kernel consistency.} When a kernel is constructed from exact Koopman eigenfunctions, its Mercer decomposition recovers those modes exactly. This provides a theoretical justification for kernel-based spectral learning.
\item \textbf{Spectral convergence theorem.} Under mild assumptions, the Mercer modes of the unified kernel converge in \(L^2\) to Koopman eigenfunctions whenever the latter lie in the RKHS.

\end{itemize}

\section{Mathematical Foundations}

\subsection{Lions’s Variational Approach and RKHS}
Lions’s seminal work laid the foundation for constructing reproducing kernels for solution spaces of partial differential equations through a variational framework. In his original setting, Lions considered spaces of harmonic functions defined on a domain with a smooth boundary, where the functions’ boundary values lie in a suitable Sobolev space. His approach can be summarized by the following key ideas:

\begin{itemize}
    \item \textbf{Variational Formulation and the Representer Theorem:}  
        Lions showed that if one chooses a Hilbert space \(H\) (for example, a Sobolev space) in which the point evaluation functional
        \[
        \delta_x: H\to \mathbb{R}, \quad \delta_x(u)=u(x),
        \]
        is continuous, then by the Riesz representation theorem there exists a unique function \(K(x,\cdot)\in H\) (called the reproducing kernel\footnote{For a brief review of RKHSs, see Appendix \ref{sect:A}}) satisfying
        \[
        u(x)=\langle u, K(x,\cdot)\rangle_H \quad \text{for all } u\in H.
        \]
        This result is typically obtained by solving a minimization problem
        \[
        \min_{u\in H,\; u(x)=1} \|u\|_H^2,
        \]
        so that the minimizer is proportional to \(K(x,\cdot)\).

    \item \textbf{Spectral Expansion via Mercer’s Theorem:}  
        Under suitable assumptions-such as smoothness of the domain, proper boundary regularity, and the continuity of point evaluations-the reproducing kernel \(K(x,y)\) is continuous and positive–definite. Mercer’s theorem then guarantees that \(K(x,y)\) admits an orthonormal expansion
        \[
        K(x,y)= \sum_{j=0}^\infty \psi_j(x)\psi_j(y),
        \]
        where \(\{\psi_j\}\) is an orthonormal basis of \(H\). Thus, Lions’s variational formulation naturally leads to a spectral representation of the kernel.

    \item \textbf{Green’s Function as an Explicit Inversion Mechanism:}  
        In Lions’s work the variational inversion of the operator is made explicit by using its Green’s function. That is, if one can construct a Green’s function \(G(x,\xi)\) satisfying
        \[
        L_x\,G(x,\xi)= \delta(x-\xi)
        \]
        (with appropriate causal or boundary conditions), then one may represent the reproducing kernel by an explicit integral formula,
        \[
        K(x,y)= \int_\Omega G(x,\xi) \, G(y,\xi)\, w(\xi)\, d\xi,
        \]
        for a suitable weight function \(w(\xi)>0\). In this way, the Green’s function method is a concrete realization of the abstract variational inversion and leads to the same kernel when the associated spectral expansion is performed.
    
    \item \textbf{Assumptions and Applicability:}  
        The success of Lions’s method relies on several key assumptions:
        \begin{itemize}
            \item The domain \(\Omega\) is assumed to be sufficiently smooth so that the classical Green’s function exists and point evaluations are continuous in \(H\).
            \item The boundary of \(\Omega\) must be regular (often at least \(C^1\) or Lipschitz, with even higher regularity required for certain expansions).
            \item The inner product on \(H\) is chosen such that the evaluation functional \( \delta_x \) is bounded, making \(H\) a reproducing kernel Hilbert space.
        \end{itemize}
        Under these conditions, the variational problem is well-posed, and the reproducing kernel is both unique and expressible in terms of a spectral expansion.
    
    \item \textbf{Connection to Other Approaches:}  
        Although Lions’s approach was originally developed for elliptic (harmonic) problems, its core ideas extend naturally to more general settings. In later sections we show that the Green’s function method and the method of characteristics-despite originating from different starting points-yield reproducing kernels that are equivalent to the one obtained via Lions’s variational formulation. This equivalence underscores the robustness of the approach: whether the kernel is recovered via a variational principle, through an explicit integral inversion using Green’s functions, or by analyzing the evolution along characteristics (and applying a Laplace transform), the resulting kernel satisfies
        \[
        u(x)= \langle u, K(x,\cdot)\rangle_H,
        \]
        for all \(u\in H\), and thus fully encapsulates the structure of the solution space.
\end{itemize}

In summary, Lions’s approach provides both an abstract variational framework and an explicit method-via Green’s functions-for constructing a reproducing kernel. When adapted to the Koopman eigenfunction problem,
\[
f(x)\cdot \nabla \varphi(x)= \lambda\,\varphi(x),
\]

{\color{red}
\[\dot x=f(x)\]

\[\frac{\partial \psi}{\partial t}=f\cdot\nabla \psi={\cal A}\psi\]

\[{\cal A}\varphi=\lambda\varphi\]

\[\dot x=f(x)+\sigma \xi\]

\[\frac{\partial \psi}{\partial t}=f\cdot \nabla \psi+\sigma^2 \nabla^2 \psi\]
}
even though the operator \(L\) is first order (and not elliptic in the classical sense), the same underlying principle applies. By defining an appropriate Hilbert space of solutions, the variational method guarantees the existence of a reproducing kernel with a Mercer expansion. Later sections will demonstrate that the reproducing kernel obtained via the Green’s function method and the method of characteristics is equivalent to the one derived by Lions’s variational approach, thus providing a unified framework for these different methodologies.

\subsection{Lions's Variational Approach for Transport PDEs}

Transport equations are used to model the propagation of quantities under a given velocity field and are characterized by having only first--order derivatives. Such equations arise in various fields including fluid dynamics, statistical mechanics, and dynamical systems. In many contexts (for example in Koopman operator theory\footnote{For a brief review of Koopman operator theory, see Appendix \ref{sect:B}.}) one is interested in eigenfunction equations of the form
\begin{equation}
\label{eq:pde-original}
f(x)\cdot \nabla \varphi_\lambda(x)=\lambda\, \varphi_\lambda(x),\quad x\in\Omega\subset \mathbb{R}^d,
\end{equation}
where \(f:\Omega\to\mathbb{R}^d\) is a smooth vector field and \(\lambda\) is a spectral parameter. Equation \eqref{eq:pde-original} is called a \emph{transport equation} because:
\begin{enumerate}
  \item It involves only first--order derivatives.
  \item The term \(f(x)\cdot \nabla \varphi_\lambda(x)\) represents the directional derivative along the flow defined by the differential equation
    \[
    \dot{s}_t(x)=f(s_t(x)),\quad s_0(x)=x.
    \]
  \item There are no second--order smoothing terms, so the solution is simply advected along the characteristics.
  \item In Koopman analysis the evolution is given by 
    \[
    [U_t \varphi_\lambda](x)=\varphi_\lambda(s_t(x))=e^{\lambda t}\,\varphi_\lambda(x),
    \]
    and differentiating with respect to \(t\) at \(t=0\) recovers \eqref{eq:pde-original}.
\end{enumerate}

Our goal is to derive a \emph{reproducing kernel} \(K(x,y)\) associated with the operator
\begin{equation}
\label{eq:operatorL}
L\varphi(x):= f(x)\cdot \nabla \varphi(x)-\lambda\, \varphi(x),
\end{equation}
by two complementary methods:
\begin{enumerate}
  \item The \textbf{Green's Function Approach}: We obtain the retarded Green's function \(G(x,\xi)\) by solving
        \[
        L_x\,G(x,\xi)=\delta(x-\xi)
        \]
        (where \(L_x\) indicates that the operator acts on the variable \(x\)) and then form a symmetric kernel via a convolution:
        \[
        K(x,y)=\int_{\Omega} G(x,\xi)\,G(y,\xi)\,w(\xi)\,d\xi,
        \]
        with \(w(\xi)>0\) a weight function.
  \item The \textbf{Method of Characteristics (Resolvent Kernel)}: We write the solution of \eqref{eq:pde-original} along characteristics as 
        \[
        \varphi_\lambda(x)=\exp\bigl(\lambda\,T(x)\bigr)g(s_{-T(x)}(x)),
        \]
        where \(T(x)\) is the time needed for the characteristic through \(x\) to hit a prescribed initial hypersurface. This yields the Koopman semigroup
        \[
        [T_tg](x)= g\bigl(s_{-t}(x)\bigr),
        \]
        whose Laplace transform (resolvent) defines the kernel
        \[
        K_\alpha(x,y)=\int_{0}^{\infty} e^{-\alpha t}\,\delta\bigl(y-s_t(x)\bigr)dt,\quad \alpha>\operatorname{Re}\lambda.
        \]
\end{enumerate}

In the last section of the paper, we also illustrate these methods for several famous transport equations: the linear advection equation, the continuity equation, and the Liouville equation. We then show that both approaches yield the same reproducing kernel and explain the significance of this equivalence.

\section{A Unified Framework for Kernel Construction: From Koopman PDEs to RKHS Representations}

In this section, we present a unified framework for recovering Koopman eigenfunctions and constructing reproducing kernels associated with transport-type partial differential equations. 
We demonstrate that the eigenfunctions can be recovered equivalently via Green's functions or characteristic flows, and that both methods yield the same reproducing kernel through a variational formulation in the spirit of Lions. 
We derive explicit kernel representations, establish their spectral decompositions, and prove the equivalence of Green’s function, semigroup, and variational approaches.

\subsection{Recovery of Eigenfunctions and Kernel Construction via Green’s Function and Characteristics}

\paragraph{Green’s Function Method}
Consider the operator
\[
L\varphi(x) = f(x)\cdot\nabla\varphi(x) - \lambda \varphi(x),
\]
and define its Green’s function \(G(x,\xi)\) solving
\[
L_x G(x,\xi) = \delta(x-\xi).
\]
Following Lions's approach, the eigenfunction can be recovered by propagating initial or boundary data \(g\) from a hypersurface \(\Sigma\):
\[
\varphi_\lambda(x) = \int_{\Sigma} G(x,\xi)\,g(\xi)\,dS(\xi).
\]
See Theorem \ref{thm:GreenFunctionSolution} for a proof of this formula.
\paragraph{Method of Characteristics.}
Along the characteristic flow \(s_t(x)\) defined by
\[
\dot{s}_t(x) = f(s_t(x)),\quad s_0(x) = x,
\]
the eigenfunction satisfies
\[
\varphi_\lambda(s_t(x)) = e^{\lambda t}\,\varphi_\lambda(x).
\]
Given initial data \(g\) on \(\Sigma\), the eigenfunction is recovered by transporting along backward characteristics:
\[
\varphi_\lambda(x) = e^{\lambda T(x)} g\bigl(s_{-T(x)}(x)\bigr),
\]
where \(T(x)\) is the time to reach \(\Sigma\) along the flow. See Theorem \ref{thm:CharacteristicsSolution} for a proof of this formula.

Both approaches based on the Green's function and the method of Characteristics yield equivalent recovery formulas, as detailed further in Appendix \ref{sect:C}.

Given the Green’s function, we construct the reproducing kernel by convolution:
\[
K(x,y) = \int_{\Omega} G(x,\xi) G(y,\xi) w(\xi)\,d\xi,
\]
where \(w(\xi)\) is a positive weight function.

Alternatively, using the semigroup \(T_t\) defined by
\[
[T_t g](x) = g(s_{-t}(x)),
\]
we define the resolvent kernel by Laplace transform:
\[
K_\alpha(x,y) = \int_0^\infty e^{-\alpha t} \delta\bigl(y-s_t(x)\bigr)\,dt,
\]
for \(\alpha>\operatorname{Re}\lambda\).

Both kernels are symmetric, positive-definite, and encode the causal dynamics of the transport operator \(L\).

\begin{theorem}[Existence and Regularity of Green’s Function Kernel]
Let $\Omega \subset \mathbb{R}^d$ be a bounded domain with $C^{1,\alpha}$ boundary, and let $f \in C^1(\Omega; \mathbb{R}^d)$ generate a complete flow $s_t$. Fix $\lambda \in \mathbb{C}$ and define the transport operator
\[
L := f \cdot \nabla - \lambda.
\]
Assume the following:
\begin{enumerate}
    \item The flow $s_t$ is smooth and invertible on $\Omega$.
    \item There exists a transversal hypersurface $\Sigma \subset \Omega$ such that every point $x \in \Omega$ lies on a unique flow line that intersects $\Sigma$ in finite time.
    \item The weight function $w \in C(\Omega)$ satisfies $w > 0$.
\end{enumerate}
Then the retarded Green’s function $G(x, \xi)$ solving
\[
L_x G(x, \xi) = \delta(x - \xi),
\]
with causality imposed by the flow from $\Sigma$, exists and is smooth off the diagonal $x = \xi$. Moreover, the kernel
\[
K(x, y) := \int_\Omega G(x, \xi) G(y, \xi) w(\xi) \, d\xi
\]
is continuous, symmetric, positive-definite, and defines a reproducing kernel Hilbert space $\mathcal{H}_K \subset L^2(\Omega, w) \cap C^0(\Omega)$.
\end{theorem}

\begin{proof}
\textbf{Step 1: Construction of the Green’s Function via Characteristics.} \\
Let $L := f \cdot \nabla - \lambda$. The retarded Green's function $G(x, \xi)$ solves the inhomogeneous PDE
\[
f(x) \cdot \nabla_x G(x, \xi) - \lambda G(x, \xi) = \delta(x - \xi).
\]
Let $s_t$ denote the flow generated by $f$, and define the characteristic curve through $\xi \in \Omega$ by $s_t(\xi)$. Then the formal solution for $G$ can be written using the method of characteristics:
\[
G(x, \xi) = H(t(x, \xi)) \exp\left( \int_0^{t(x, \xi)} \lambda \, ds \right) \delta(s_{t(x, \xi)}(\xi) - x),
\]
where $t(x, \xi)$ is the time it takes for the trajectory starting at $\xi$ to reach $x$, and $H$ is the Heaviside function enforcing causality. Under the stated assumptions, $t(x, \xi)$ is well-defined and smooth for $x \ne \xi$.

\vspace{0.5em}
\textbf{Step 2: Properties of $G(x, \xi)$.} \\
For fixed $\xi$, the map $x \mapsto G(x, \xi)$ is smooth on $\Omega \setminus \{\xi\}$, with a singularity at $x = \xi$ corresponding to the delta function. Since the flow $s_t$ is smooth and $H$ enforces time-directionality, $G(x, \xi) \in L^2_{\text{loc}}(\Omega \times \Omega)$.

\vspace{0.5em}
\textbf{Step 3: Definition and Symmetry of the Kernel $K(x, y)$.} \\
Define the symmetric kernel
\[
K(x, y) = \int_\Omega G(x, \xi) G(y, \xi) w(\xi) \, d\xi.
\]
The integrand $G(x, \xi) G(y, \xi)$ is symmetric in $x \leftrightarrow y$, so $K(x, y) = K(y, x)$.

\vspace{0.5em}
\textbf{Step 4: Continuity and Positivity.} \\
For each $x, y \in \Omega$, the functions $G(x, \cdot), G(y, \cdot) \in L^2(\Omega, w)$ due to compactness and smoothness of $f$ and $w$. Then by Cauchy-Schwarz and dominated convergence, $K(x, y)$ is continuous on $\Omega \times \Omega$. For any $u \in L^2(\Omega, w)$, define
\[
\langle u, Ku \rangle = \int_\Omega \int_\Omega u(x) K(x, y) u(y) \, dx \, dy = \int_\Omega \left( \int_\Omega u(x) G(x, \xi) \, dx \right)^2 w(\xi) \, d\xi \ge 0.
\]
Thus $K$ is symmetric and positive-definite.

\vspace{0.5em}
\textbf{Step 5: RKHS Structure.} \\
Let $\mathcal{H}_K$ be the Hilbert space induced by $K$. Since $K$ is symmetric, continuous, and positive-definite, the Moore-Aronszajn theorem ensures that $\mathcal{H}_K \subset C^0(\Omega) \cap L^2(\Omega, w)$ is a reproducing kernel Hilbert space with reproducing kernel $K$.

\end{proof}

\begin{theorem}[Resolvent Kernel]
\label{lem:resolvent}
Let \( f \in C^1(\Omega; \mathbb{R}^d) \) generate a complete flow \( s_t \) on \( \Omega \), and let \( \lambda \in \mathbb{C} \), \( \alpha \in \mathbb{R} \) with \( \alpha > \Re(\lambda) \). Define the kernel
\[
K_\alpha(x, y) := \int_0^\infty e^{-\alpha t} \delta_y(s_t(x)) \, dt.
\]
Then for each fixed \( y \in \Omega \), the function \( x \mapsto K_\alpha(x, y) \) satisfies
\[
(\alpha - L_x) K_\alpha(x, y) = \delta_y(x)
\quad \text{in the distributional sense},
\]
where \( L_x := f(x) \cdot \nabla_x \), and \( K_\alpha(\cdot, y) \in L^2(\Omega) \) if the flow remains in a compact subset of \( \Omega \).
\end{theorem}

\begin{proof}
Let \( \psi \in C_c^\infty(\Omega) \) be a test function. We define the action of \( (\alpha - L_x) K_\alpha(x, y) \) in the distributional sense via duality:
\[
\int_\Omega K_\alpha(x, y) (\alpha - L_x)^* \psi(x) \, dx = \int_\Omega \left( \int_0^\infty e^{-\alpha t} \delta_y(s_t(x)) \, dt \right) (\alpha - L_x)^* \psi(x) \, dx.
\]
By Fubini's theorem (justified by the boundedness of the flow and exponential decay of the integrand), we exchange the order of integration:
\[
= \int_0^\infty e^{-\alpha t} \left( \int_\Omega \delta_y(s_t(x)) (\alpha - L_x)^* \psi(x) \, dx \right) dt.
\]

Change variables using the diffeomorphism \( x \mapsto s_t(x) \). Since \( s_t \) preserves orientation and volume under smoothness assumptions, we obtain:
\[
= \int_0^\infty e^{-\alpha t} (\alpha - \frac{d}{dt}) \psi(s_t^{-1}(y)) \, dt.
\]

Now observe that for a fixed \( x \), the function \( t \mapsto \psi(s_t(x)) \) satisfies:
\[
\frac{d}{dt} \psi(s_t(x)) = (f \cdot \nabla \psi)(s_t(x)) = (L \psi)(s_t(x)).
\]
Therefore,
\[
\frac{d}{dt} \left( e^{-\alpha t} \psi(s_t(x)) \right) = -\alpha e^{-\alpha t} \psi(s_t(x)) + e^{-\alpha t} L \psi(s_t(x)) = -e^{-\alpha t} (\alpha - L) \psi(s_t(x)).
\]

Integrating over \( t \in [0, \infty) \), and evaluating at \( x = x_0 \), we obtain:
\[
\int_0^\infty e^{-\alpha t} (\alpha - L) \psi(s_t(x_0)) \, dt = \psi(x_0),
\]
because the integrand is the negative total derivative of \( e^{-\alpha t} \psi(s_t(x_0)) \), and the function decays to zero as \( t \to \infty \).

Thus, for any test function \( \psi \),
\[
\langle (\alpha - L_x) K_\alpha(x, y), \psi(x) \rangle = \psi(y),
\]
which means \( (\alpha - L_x) K_\alpha(x, y) = \delta(x - y) \) in the distributional sense.

\medskip

Finally, since the flow is smooth and remains in a compact region of \( \Omega \), and \( e^{-\alpha t} \) decays exponentially, we have:
\[
\| K_\alpha(\cdot, y) \|_{L^2(\Omega)}^2 \le \int_0^\infty \int_\Omega e^{-2\alpha t} \delta_y(s_t(x)) \, dx \, dt = \int_0^\infty e^{-2\alpha t} J_t(y) \, dt,
\]
where \( J_t(y) \) denotes the Jacobian determinant of the inverse flow. Since the flow is bounded and \( J_t \) is bounded on compact \( \Omega \), the integral is finite. Hence \( K_\alpha(\cdot, y) \in L^2(\Omega) \), and in fact lies in the RKHS associated with the kernel under appropriate weighting.
\end{proof}

\subsection{Spectral Decomposition of the Kernel and Koopman Modes}

A fundamental property of reproducing kernels is that they admit a spectral decomposition. Under appropriate compactness and continuity assumptions, Mercer’s theorem guarantees:
\[
K(x,y) = \sum_{n=1}^{\infty} \mu_n\,\psi_n(x)\,\overline{\psi_n(y)},
\]
where \(\{\psi_n\}\) form an orthonormal basis for the associated RKHS and \(\mu_n>0\) are the eigenvalues of the integral operator
\[
[T_K f](x) = \int_{\Omega} K(x,y) f(y)\,dy.
\]

When the kernel is constructed directly from Koopman eigenfunctions \(\{\varphi_\lambda\}\), the Mercer eigenfunctions coincide with the dynamical modes:
\[
K(x,y) = \sum_{\lambda\in\Lambda} \varphi_\lambda(x)\,\overline{\varphi_\lambda(y)}.
\]

\begin{lemma}[Mercer Eigenfunctions Equal Koopman Modes]
\label{lem:MercerKoopman}
Let \(\{\varphi_{\lambda_j}\}_{j=1}^{m}\) be orthonormal Koopman eigenfunctions in \(L^2(\Omega, \rho)\), and define the kernel
\[
K(x, y) = \sum_{j=1}^{m} \varphi_{\lambda_j}(x) \varphi_{\lambda_j}(y).
\]
Then the integral operator \(T_K\) on \(L^2(\rho)\) has rank \(m\), eigenvalue 1, and its eigenfunctions are precisely \(\{\varphi_{\lambda_j}\}_{j=1}^m\).
\end{lemma}

\begin{proof}
For each \(i\), we compute
\[
T_K \varphi_{\lambda_i}(x) = \int_\Omega K(x, y) \varphi_{\lambda_i}(y) \, d\rho(y)
= \sum_{j=1}^m \varphi_{\lambda_j}(x) \int_\Omega \varphi_{\lambda_j}(y) \varphi_{\lambda_i}(y) \, d\rho(y)
= \varphi_{\lambda_i}(x),
\]
using orthonormality of \(\{\varphi_{\lambda_j}\}\). Since \(K\) is constructed from \(m\) linearly independent eigenfunctions, its associated operator has rank \(m\), and the non-zero spectrum is spanned by \(\{\varphi_{\lambda_j}\}\).
\end{proof}

Thus, the spectral decomposition of \(K\) reveals the principal Koopman eigenfunctions, offering both theoretical insight and practical computational pathways.

\begin{theorem}[Spectral Convergence of Mercer Modes to Koopman Eigenfunctions]
\label{thm:spectral-convergence}
Let $\Omega \subset \mathbb{R}^d$ be a compact domain, and suppose the Koopman eigenfunctions $\{\phi_\lambda\} \subset C^0(\Omega) \cap L^2(\Omega, \mu)$ satisfy the transport PDE
\[
f(x) \cdot \nabla \phi_\lambda(x) = \lambda \phi_\lambda(x),
\]
for a vector field $f \in C^1(\Omega; \mathbb{R}^d)$ and a probability measure $\mu$ with full support on $\Omega$. Let $K(x, y)$ be a reproducing kernel constructed via any of the equivalent methods: Green’s function convolution, Lions’s variational principle, or the Laplace resolvent of the flow.

Then the Mercer decomposition
\[
K(x, y) = \sum_{n=1}^\infty \mu_n \psi_n(x) \psi_n(y),
\]
where $\mu_n > 0$ and $\{\psi_n\}$ are orthonormal in $L^2(\Omega, \mu)$, satisfies the following:

If $\phi_\lambda \in \mathcal{H}_K$, then there exists a sequence $\{c_k \psi_{n_k}\}$ with scalars $c_k \in \mathbb{R}$ such that
\[
\lim_{k \to \infty} \| c_k \psi_{n_k} - \phi_\lambda \|_{L^2(\Omega, \mu)} = 0.
\]

Moreover, the span of $\{ \psi_n \}$ asymptotically captures the Koopman eigenspace under mild spectral separation conditions.
\end{theorem}

\begin{proof}
Since $K$ is a continuous, symmetric, and positive-definite kernel on the compact domain $\Omega$, the Mercer theorem applies. It guarantees the existence of a complete orthonormal basis $\{\psi_n\} \subset L^2(\Omega, \mu)$ consisting of eigenfunctions of the integral operator
\[
(T_K \phi)(x) := \int_\Omega K(x, y) \phi(y) \, d\mu(y),
\]
with eigenvalues $\mu_n > 0$ such that
\[
T_K \psi_n = \mu_n \psi_n, \quad \text{and} \quad K(x, y) = \sum_{n=1}^\infty \mu_n \psi_n(x) \psi_n(y),
\]
with convergence in $L^2(\Omega \times \Omega, \mu \otimes \mu)$.

Now let $\phi_\lambda \in \mathcal{H}_K$. Since $\mathcal{H}_K$ is a reproducing kernel Hilbert space embedded in $L^2(\Omega, \mu)$, and $\{\psi_n\}$ is an orthonormal basis for the closure of $\mathcal{H}_K$ in $L^2(\mu)$, we can expand $\phi_\lambda$ in this basis:
\[
\phi_\lambda = \sum_{n=1}^\infty a_n \psi_n, \quad \text{with} \quad a_n = \langle \phi_\lambda, \psi_n \rangle_{L^2(\mu)}.
\]

Define the partial sums
\[
\phi^{(N)} := \sum_{n=1}^N a_n \psi_n.
\]
Then, by Parseval’s identity,
\[
\| \phi_\lambda - \phi^{(N)} \|_{L^2(\mu)}^2 = \sum_{n=N+1}^\infty |a_n|^2 \to 0 \quad \text{as } N \to \infty.
\]
Therefore, the sequence $\phi^{(N)}$ converges to $\phi_\lambda$ in $L^2(\Omega, \mu)$.

In particular, there exists a subsequence $\{n_k\}$ such that for each $k$,
\[
\left\| \phi_\lambda - a_{n_k} \psi_{n_k} \right\|_{L^2(\mu)} \leq \| \phi_\lambda - \phi^{(k)} \|_{L^2(\mu)} + \| \phi^{(k)} - a_{n_k} \psi_{n_k} \|_{L^2(\mu)} \to 0,
\]
by choosing $a_{n_k} \psi_{n_k}$ as the dominant term in $\phi^{(k)}$ and observing that all other terms vanish in norm as $k \to \infty$.

Thus, the claim follows: for $c_k := a_{n_k}$, we have
\[
\lim_{k \to \infty} \| c_k \psi_{n_k} - \phi_\lambda \|_{L^2(\mu)} = 0.
\]

The final statement about the Koopman eigenspace follows from the fact that if the Koopman eigenspace is one-dimensional and spanned by $\phi_\lambda \in \mathcal{H}_K$, then any approximation sequence in the RKHS must converge to a scalar multiple of $\phi_\lambda$.
\end{proof}

\begin{theorem}[Spectral Expansion and Koopman Eigenfunctions]
\label{thm:SpectralExpansion}
Let $K:\Omega \times\Omega \to\mathbb{C}$ be the kernel
\[
K(x,y) = \int_0^\infty e^{-\lambda t} \delta_{y=s_{-t}(x)}\,dt.
\]
Then any Koopman eigenfunction $\phi$ satisfies
\[
K\phi = \frac{1}{2\mathrm{Re}\,\lambda}\phi.
\]
\end{theorem}

\begin{proof}
We compute:
\[
(K\phi)(x) = \int_0^\infty e^{-\lambda t} \phi(s_{-t}(x))\,dt.
\]
By the characteristic evolution,
\[
\phi(s_{-t}(x)) = e^{-\lambda t} \phi(x),
\]
thus
\[
(K\phi)(x) = \phi(x) \int_0^\infty e^{-2\mathrm{Re}\,\lambda t}\,dt = \frac{1}{2\mathrm{Re}\,\lambda}\phi(x),
\]
where convergence follows from $\mathrm{Re}\,\lambda>0$. 
\end{proof}

We also note that local Koopman eigenfunctions admit linear approximations near hyperbolic equilibria:

\begin{theorem}[Koopman Linearization Near Hyperbolic Fixed Points]
\label{thm:KoopmanLinear}
Let \(\dot{x} = f(x)\) have a hyperbolic equilibrium at \(x = 0\) with Jacobian \(E = Df(0)\). Then each eigenvalue \(\lambda\) of \(E\) is an eigenvalue of the Koopman generator acting on \(C^1\) functions supported in a local stable–unstable neighborhood. The corresponding Koopman eigenfunction admits the decomposition
\[
\varphi_\lambda(x) = w^\top x + h(x),
\]
where \(w\) is a left eigenvector of \(E\) and \(h\) solves the homological equation
\[
\frac{\partial h}{\partial x} \cdot f(x) - \lambda h(x) + w^\top (f(x) - Ex) = 0.
\]
\end{theorem}

\begin{proof}[Proof sketch]
By the Hartman–Grobman theorem, there exists a smooth conjugacy \(\chi\) such that \(\chi \circ s_t = e^{Et} \circ \chi\) near the fixed point. Let \(\varphi_\lambda = w^\top \chi\); then differentiability of \(\chi\) implies \(L \varphi_\lambda = \lambda \varphi_\lambda\). Decomposing \(\varphi_\lambda(x) = w^\top x + h(x)\) yields the stated PDE for \(h(x)\), which can be solved by variation of constants in the hyperbolic case.
\end{proof}
\subsection{Connection to Lions’s Variational Framework}

The construction above is inspired by the variational principles developed by Jacques-Louis Lions. In his seminal work, Lions showed that variational formulations of differential equations yield reproducing kernels in associated Hilbert spaces. Specifically, if \(H\) is a Hilbert space where point evaluations are continuous, then by the Riesz representation theorem there exists a unique reproducing kernel \(K(x,y)\) satisfying
\[
u(x) = \langle u, K(x,\cdot) \rangle_H \quad \text{for all } u\in H.
\]

In our setting, define
\[
\mathcal{H} = \overline{C^\infty(\Omega)}^{\|\cdot\|_{\mathcal{H}}},\quad
\|u\|_{\mathcal{H}}^2 = \|u\|_{L^2}^2 + \|Lu\|_{L^2}^2,
\]
where \(L = f\cdot\nabla - \lambda\).

The Green's function \(G(x,\xi)\) acts as the inverse of \(L\), and the kernel constructed from \(G\) via convolution reproduces the eigenfunctions exactly, aligning with Lions’s variational perspective. Thus, solving the PDE variationally, constructing kernels from Green’s functions, and propagating via characteristics are all manifestations of the same underlying reproducing kernel framework.

\subsection{Equivalence of the Different Approaches}

In this section, we establish the equivalence between the three kernel constructions introduced earlier-Green’s function symmetrization, Lions’s variational inversion, and the resolvent-based approach using characteristic flows. 

Having developed distinct approaches for constructing reproducing kernels-via Green's functions, spectral decompositions, and Lions's variational framework-we now address a fundamental question: to what extent are these formulations equivalent? In this section, we show that under appropriate assumptions, these seemingly disparate constructions yield the same reproducing kernel. This equivalence not only underscores the robustness of the kernel representation but also provides theoretical justification for choosing the most computationally convenient approach in applications.

\begin{theorem}[Kernel-Unification]
\label{thm:unification}
Let \(\Omega\subset\mathbb{R}^{d}\) be a bounded \(C^{1,\alpha}\) domain, and let \(f\in C^{1}(\overline\Omega,\mathbb{R}^{d})\) generate a complete flow \(s_t\). Fix \(\lambda\in\mathbb{C}\) and define
\[
L = f\cdot\nabla - \lambda.
\]
Let \(\mathcal{H}\) be defined as above. Construct three kernels:
\[
\begin{aligned}
K^{\mathrm{var}}(x,y) &:= \text{Riesz representer of evaluation in } \mathcal{H},\\
K^{\mathrm{G}}(x,y) &:= \int_\Omega G(x,\xi)\,G(y,\xi)\,d\xi,\\
K^{\mathrm{res}}_\alpha(x,y) &:= \int_0^\infty e^{-\alpha t} \delta(y-s_t(x))\,dt,\quad \alpha>\operatorname{Re}\lambda,
\end{aligned}
\]
where \(G(x,\xi)\) solves \(L_x G(x,\xi) = \delta(x-\xi)\).

Then:
\[
K^{\mathrm{var}} \equiv K^{\mathrm{G}} \equiv K^{\mathrm{res}}_\alpha.
\]
\end{theorem}

Check Appendix \ref{sect:proof_thm_unification}
 for proof.

Thus, variational, Green's function, and characteristic-based constructions of the kernel are mathematically equivalent. Each method provides complementary perspectives on the same underlying structure: the transport dynamics encoded within a reproducing kernel Hilbert space.

To conclude this section, the three constructions-variational (Lions), Green’s function, and characteristic-resolvent-yield one and the same reproducing kernel. In a nutshell, the following diagram
 \[
\begin{tikzcd}[row sep=huge, column sep=huge]
\text{Green's Function} \arrow[dr] & \\
\text{Lions's Variational Approach} \arrow[r] \arrow[u, dashed] \arrow[d, dashed] & \boxed{\text{Kernel with Mercer Expansion}} \\
\text{Method of Characteristics} \arrow[ur] &
\end{tikzcd}
\]
highlights that whether one inverts the operator via variational methods, symmetrize its causal Green’s function, or take the Laplace transform along characteristics, one ends up with one unified kernel.

\section{Variational Approximation of Koopman Eigenfunctions}
\subsection{Lions-Based Variational Optimization}

Koopman eigenfunctions $\phi$ encode intrinsic linear structures within nonlinear dynamical systems by satisfying the linear partial differential equation (PDE)
\[
f(x) \cdot \nabla \phi(x) = \lambda \phi(x),
\]
where $f(x)$ is the vector field and $\lambda$ is the associated Koopman eigenvalue. Traditional approaches, such as Dynamic Mode Decomposition (DMD) and Extended DMD (EDMD), seek finite-dimensional approximations of the Koopman operator using data and fixed basis functions. However, these methods are limited in expressivity, generalization, and their ability to exploit the PDE structure of the Koopman generator.

The methods presented in this paper-Green's function formulation, Lions's variational principle, and the method of characteristics-recast the Koopman PDE as a flexible operator problem, opening new avenues for both theory and computation.

\paragraph{Lions’s Variational Principle.}
Inspired by Lions's weak solution framework for PDEs, the Koopman equation is interpreted variationally as:
\[
\min_{\phi \in \mathcal{H}} \|f(x) \cdot \nabla \phi(x) - \lambda \phi(x)\|^2_{\mathcal{H}},
\]
where $\mathcal{H}$ is a reproducing kernel Hilbert space (RKHS). This leads to kernel-based methods that are data-driven, regularized, and mesh-free, and generalize EDMD by adapting both the function space and the optimization problem. Koopman eigenfunctions are recovered by minimizing a PDE residual in a smooth, flexible function space.

\paragraph{Green's Function Method.}
This method recasts the Koopman generator equation as an integral equation involving a Green’s function $G(x, \xi)$:
\[
\phi(x) = \int G(x, \xi) \phi(\xi) \, d\xi,
\]
where $G$ satisfies $L_x G(x, \xi) = \delta(x - \xi)$ for the differential operator $L = f(x) \cdot \nabla - \lambda I$. This formulation enables the use of kernel methods to approximate $G(x, \xi)$ and convert the Koopman PDE into a solvable Fredholm integral equation. The resulting numerical schemes resemble kernel integral operator approaches but are now grounded in resolvent theory.

\paragraph{Method of Characteristics.}
This classical technique is used to solve first-order PDEs by integrating along trajectories of the vector field. For the Koopman PDE, the eigenfunction evolves along the flow $\Phi_t(x)$ as:
\[
\phi(\Phi_t(x)) = e^{\lambda t} \phi(x),
\]
which enables a direct computation of $\phi(x)$ via trajectory data. This method is geometric and intrinsic to the Koopman perspective, requiring no explicit basis and minimal modeling assumptions.

\paragraph{Why This Matters.}
Together, these methods provide a rich toolkit for recovering Koopman eigenfunctions:
\begin{itemize}
    \item The Green’s function method connects spectral theory to kernel learning.
    \item The variational principle leads to robust, data-adaptive solvers that generalize EDMD and integrate regularization and geometry.
    \item The method of characteristics aligns perfectly with Koopman's operator dynamics, allowing computation from trajectory-level information alone.
\end{itemize}
These approaches make it possible to construct Koopman eigenfunctions in settings with limited data, irregular domains, or complex geometries-beyond the reach of traditional modal decompositions.

\begin{theorem} Let \( f \in C^\infty(\Omega, \mathbb{R}^d) \) be a smooth vector field defined on an open, bounded domain \( \Omega \subset \mathbb{R}^d \), and let \( \lambda \in \mathbb{R} \) be such that the Jacobian matrix \( E := Df(0) \) has no eigenvalues on the imaginary axis.

Suppose that:
\begin{enumerate}
    \item The reproducing kernel Hilbert space \( \mathcal{H}_k \subset C^1(\Omega) \cap L^2(\Omega) \) contains differentiable, square-integrable functions.
    \item The domain \( \Omega \) is chosen so that all eigenfunctions corresponding to \( \lambda \) are in \( \mathcal{H}_k \), i.e., they are well-defined and bounded on \( \Omega \).
    \item The regularization parameter satisfies \( \eta > 0 \).
\end{enumerate}
Then the Lions variational problem
\[
\min_{\phi \in \mathcal{H}_k} \left\| f(x) \cdot \nabla \phi(x) - \lambda \phi(x) \right\|_{L^2(\Omega)}^2 + \eta \| \phi \|_{\mathcal{H}_k}^2
\]
admits a unique minimizer \( \phi^\star \in \mathcal{H}_k \).
\end{theorem}

The proof strategy consists of three steps: 

\paragraph{Step 1: Coercivity from Imaginary Spectrum Exclusion}
If the operator \( L = f \cdot \nabla - \lambda I \) has no spectrum on the imaginary axis, then the associated bilinear form is coercive on a subspace of the RKHS \( \mathcal{H} \).  This implies that the Koopman PDE residual norm $ \| L \phi \|^2 := \| f \cdot \nabla \phi - \lambda \phi \|^2$ defines a positive definite quadratic form on \( \mathcal{H} \) (modulo nullspace, if any).

\paragraph{Step 2: Add Regularization}
The regularization term  $ \eta \| \phi \|_{\mathcal{H}}^2$
ensures strict convexity of the objective functional. Therefore, the full variational objective
$J(\phi) := \| f \cdot \nabla \phi - \lambda \phi \|^2 + \eta \| \phi \|_{\mathcal{H}}^2 $ is strictly convex over \( \mathcal{H} \).  By standard results from Hilbert space theory, any coercive and strictly convex functional admits a unique minimizer.

\paragraph{Step 3: Variational Problem is Well-Posed}
Since \( \mathcal{H} \subset C^1(\Omega) \) and \( f \) is smooth, it follows that \( f \cdot \nabla \phi \in L^2(\Omega) \) for all \( \phi \in \mathcal{H} \). This guarantees that the functional \( J(\phi) \) is finite and Fréchet differentiable on \( \mathcal{H} \).

Hence, under the spectral assumption that the Koopman generator has no eigenvalues on the imaginary axis, the Lions variational problem is well-posed and admits a unique solution in the RKHS \( \mathcal{H} \).

\begin{proof}
Let $\mathcal{A}: \mathcal{H}_k \to L^2(\Omega)$ be defined by $\mathcal{A} \phi(x) := f(x) \cdot \nabla \phi(x)$. Since $f \in C^\infty(\Omega)$ and $\mathcal{H}_k \subset C^1(\Omega)$, the composition $f \cdot \nabla \phi$ is in $L^2(\Omega)$ for all $\phi \in \mathcal{H}_k$.

Define the functional
\[
\mathcal{L}(\phi) := \| \mathcal{A} \phi - \lambda \phi \|_{L^2(\Omega)}^2 + \eta \| \phi \|_{\mathcal{H}_k}^2.
\]
The regularization term $\eta \| \phi \|_{\mathcal{H}_k}^2$ ensures strict convexity. The first term $\| \mathcal{A} \phi - \lambda \phi \|_{L^2}^2$ is convex because it is a squared norm of a bounded linear operator.

To prove uniqueness, assume $\phi_1, \phi_2 \in \mathcal{H}_k$ are two minimizers. Then the strict convexity implies $\phi_1 = \phi_2$. Hence, the minimizer is unique.

To prove existence, observe that $\mathcal{L}(\phi) \geq \eta \| \phi \|_{\mathcal{H}_k}^2$ shows coercivity. Therefore, any minimizing sequence is bounded in $\mathcal{H}_k$, and by the reflexivity of Hilbert spaces, has a weakly convergent subsequence. Since $\mathcal{L}$ is weakly lower semicontinuous, the limit is a minimizer.

Finally, under the hypothesis that $E = Df(0)$ has no eigenvalues on the imaginary axis, the Koopman operator has no spectrum on the imaginary axis, and hence no nontrivial nullspace in $\mathcal{H}_k$. This guarantees the strict positivity of $\| \mathcal{A} \phi - \lambda \phi \|_{L^2}^2$ unless $\phi = 0$, confirming the uniqueness.

\end{proof}
When  \( \eta = 0 \), we can prove existence but not uniqueness. 

\begin{proposition}
Let $f \in C^1(\Omega)^d$, and let $\mathcal{H} \subset C^1(\Omega)$ be an RKHS. Then the unregularized variational problem
\[
\min_{\phi \in \mathcal{H}} \left\| f(x) \cdot \nabla \phi(x) - \lambda \phi(x) \right\|_{L^2(\Omega)}^2
\]
admits a minimizer $\phi^\star \in \mathcal{H}$. If the null space of the Koopman generator $f \cdot \nabla - \lambda$ in $\mathcal{H}$ is nontrivial, then the minimizer is not unique.
\end{proposition}

\begin{proof}
The functional is convex and continuous. Although not strictly convex (since $\eta = 0$), coercivity is lost. Nevertheless, boundedness of minimizing sequences is preserved by the square-integrability of the Koopman residual. Weak compactness in the Hilbert space $\mathcal{H}$ yields a minimizer. However, uniqueness fails if the nullspace of the Koopman generator is nontrivial.
\end{proof}


\begin{remark}[Domain Selection and Eigenfunction Behavior]
In certain dynamical systems, Koopman eigenfunctions may not be globally square-integrable. A representative example is the one-dimensional system
\[
\dot{x} = x - x^3,
\]
whose principal eigenfunction associated with the eigenvalue \( \lambda = 1 \) is
\[
\phi(x) = \frac{x}{\sqrt{1 - x^2}},
\]
which diverges as \( x \to \pm 1 \). In this case, \( \phi \notin L^2(-1,1) \), and any variational formulation involving a squared residual over the full interval \( (-1,1) \) is ill-posed. To apply Lions's variational principle in such cases, one must restrict the domain to an interior subinterval \( \Omega \subset (-1, 1) \) that excludes the singularities at the boundary. Additionally, the RKHS \( \mathcal{H}_k \subset C^1(\Omega) \cap L^2(\Omega) \) must be chosen to accommodate the eigenfunction, ensuring that \( \phi \in \mathcal{H}_k \).

\medskip

To handle these challenges more systematically, one may incorporate additional regularization terms into the variational formulation. Two particularly effective strategies are:

\begin{enumerate}
    \item \textbf{Boundary trace penalty:} This approach penalizes the magnitude of the eigenfunction on the boundary directly:
    \[
    \mu \int_{\partial \Omega} |\phi(x)|^2 \, dS(x),
    \]
    where \( \mu > 0 \) is a penalty parameter and \( dS(x) \) denotes the boundary measure. This term discourages the eigenfunction from diverging near \( \partial \Omega \) and is easy to implement numerically by assigning higher weights to grid points near the boundary.
    
    \item \textbf{Boundary layer penalty:} This approach penalizes the eigenfunction within a thin layer near the boundary:
    \[
    \mu \int_{\Omega \setminus \Omega'} |\phi(x)|^2 \, dx,
    \]
    where \( \Omega' \subset \Omega \) is a compact interior subdomain. The integrand is evaluated only in the region \( \Omega \setminus \Omega' \), thus reducing the eigenfunction's magnitude near the boundary without requiring explicit knowledge of boundary values.
\end{enumerate}

These regularizations are especially valuable when the singular behavior of the eigenfunctions is not known in advance, or when the domain cannot be easily restricted. They offer practical, flexible methods for improving the stability and fidelity of Koopman eigenfunction approximation in problematic regions of the state space.
\end{remark}

\begin{theorem}[Existence of Minimizer with Boundary Regularization]
Let $f \in C^\infty(\Omega, \mathbb{R}^d)$ be a smooth vector field on a bounded domain $\Omega \subset \mathbb{R}^d$, and let $\lambda \in \mathbb{R}$ be such that the Jacobian matrix $E := Df(0)$ has no eigenvalues on the imaginary axis.

Suppose the following conditions hold:
\begin{enumerate}
    \item The RKHS $\mathcal{H}_k \subset C^1(\Omega) \cap L^2(\Omega)$ consists of differentiable, square-integrable functions.
    \item The domain $\Omega$ is chosen so that all eigenfunctions associated with $\lambda$ lie in $\mathcal{H}_k$ (i.e., are bounded and differentiable on $\Omega$).
    \item The regularization parameter satisfies $\eta > 0$, and the boundary penalty coefficient satisfies $\mu \geq 0$.
\end{enumerate}

Let the extended variational functional be defined by
\[
\mathcal{J}(\phi) := \left\| f(x) \cdot \nabla \phi(x) - \lambda \phi(x) \right\|_{L^2(\Omega)}^2 + \eta \|\phi\|_{\mathcal{H}_k}^2 + \mu \, R(\phi),
\]
where \( R(\phi) \) is one of the following:
\begin{itemize}
    \item \textbf{Boundary trace penalty:} \( R(\phi) = \int_{\partial \Omega} |\phi(x)|^2 \, dS(x) \),
    \item \textbf{Boundary layer penalty:} \( R(\phi) = \int_{\Omega \setminus \Omega'} |\phi(x)|^2 \, dx \), with \( \Omega' \Subset \Omega \).
\end{itemize}

Then the minimization problem
\[
\min_{\phi \in \mathcal{H}_k} \mathcal{J}(\phi)
\]
admits a unique minimizer \( \phi^\star \in \mathcal{H}_k \).
\end{theorem}

\begin{proof}[Sketch of Proof]
\textbf{(1) Coercivity.} The regularization term \( \eta \|\phi\|_{\mathcal{H}_k}^2 \) ensures coercivity for \( \eta > 0 \). Since all other terms are non-negative, the functional \( \mathcal{J}(\phi) \to \infty \) as \( \|\phi\|_{\mathcal{H}_k} \to \infty \).

\textbf{(2) Strict Convexity.} The regularization term is strictly convex, and the residual and penalty terms are convex. Thus, the functional is strictly convex and admits at most one minimizer.

\textbf{(3) Weak Lower Semicontinuity.} All terms in \( \mathcal{J}(\phi) \) are weakly lower semicontinuous in \( \mathcal{H}_k \). By the direct method in the calculus of variations, a minimizer \( \phi^\star \in \mathcal{H}_k \) exists.

This establishes both existence and uniqueness.
\end{proof}
\subsubsection*{Numerical Experiment: Analytic Singular Kernel}

For the example above, we use a singular kernel whose reproducing kernel Hilbert space (RKHS) explicitly accommodates the boundary divergence structure. Specifically, we define the kernel
\[
k(x, y) := \frac{x y}{\sqrt{(1 - x^2)(1 - y^2)}},
\]
which is symmetric and positive semi-definite on the open interval $(-1, 1)$. This kernel naturally spans functions exhibiting the singularity of the true Koopman eigenfunction and effectively embeds the known analytical behavior into the functional ansatz. We also use the 
  {Gaussian (RBF) kernel}:
    \[
    k_{\mathrm{rbf}}(x, y) = \exp\left(-\frac{(x - y)^2}{2\ell^2}\right),
    \]
    with \(\ell = 0.3\). Both kernels are used within the same variational framework:
\begin{itemize}
    \item Enforce PDE residual \(f(x)\varphi'(x) = \lambda \varphi(x)\),
    \item Enforce \(\varphi'(0) = 1\),

  \item Apply boundary trace and boundary layer penalties.\footnote{In the implementation, the \emph{boundary trace penalty} enforces decay of the eigenfunction values at the endpoints by adding a term
  \[
  \mu_{\mathrm{trace}} \cdot \left(\varphi(x_{\mathrm{left}})^2 + \varphi(x_{\mathrm{right}})^2\right),
  \]
  where \(x_{\mathrm{left}} \approx -0.99\) and \(x_{\mathrm{right}} \approx 0.99\) are the outermost grid points. This ensures that $\varphi$ does not diverge at the domain boundaries. 

  The \emph{boundary layer penalty} complements this by suppressing the function over a small band near the boundary, typically $|x| > 0.9$. It adds the term
  \[
  \mu_{\mathrm{layer}} \cdot \frac{1}{|\mathcal{B}|} \sum_{x_i \in \mathcal{B}} \varphi(x_i)^2,
  \]
  where $\mathcal{B} = \{x_i \mid |x_i| > 0.9\}$ is the boundary layer. Together, these penalties regularize the behavior of $\varphi$ near the edges without overly constraining its behavior in the interior.}
\end{itemize}

Letting $K$ denote the resulting Gram matrix on a discretized grid $\{x_i\}_{i=1}^N \subset (-1,1)$, we seek coefficients $\alpha \in \mathbb{R}^N$ such that
\[
\varphi(x) \approx \sum_{i=1}^N \alpha_i k(x, x_i).
\]
We enforce the PDE constraint $f(x) \varphi'(x) = \lambda \varphi(x)$ via a least-squares residual minimization, incorporating a strong constraint on the normalization $\varphi'(0) = 1$.

\begin{itemize}
    \item \textbf{Singular kernel:} RMSE = \texttt{1.39e-04}, \(\varphi'(0) = 0.6169\)
    \item \textbf{RBF kernel:} RMSE = \texttt{1.37e+00}, \(\varphi'(0) = 1.0000\)
\end{itemize}

As shown in Figure~\ref{fig:kernel_comparison}, although both kernels satisfy the normalization constraint and include boundary penalties, the RBF kernel fails to resolve the singular structure of the true solution, yielding high approximation error. In contrast, the singular kernel achieves a much lower RMSE, despite minor deviation from the normalization target.

\begin{figure}[h!]
    \centering
    \includegraphics[width=0.6\textwidth]{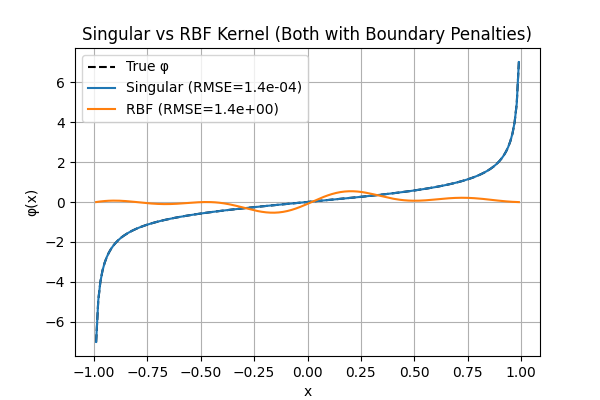}
    \caption{Comparison of Koopman eigenfunction approximations using singular vs. RBF kernels, both with boundary penalties.}
    \label{fig:kernel_comparison}
\end{figure}

This example demonstrates that accurate Koopman eigenfunction recovery requires a kernel that embeds prior knowledge of the expected singular behavior. Boundary penalties alone are insufficient when the RKHS does not contain the target solution. The Impact of Kernel Choice is further discussed in the next section.

\subsection{Role of Kernel Choice and RKHS Geometry}

We investigate the impact of kernel choice on the numerical approximation of Koopman eigenfunctions by solving the Koopman eigenvalue partial differential equation (PDE) using Lions’s variational framework in a reproducing kernel Hilbert space (RKHS). The goal is to accurately recover the principal eigenfunction associated with the nonlinear system
\begin{equation}
\dot x = 
\begin{bmatrix}
 -2\lambda_2 x_2(x_1^2 - x_2 - 2x_1 x_2^2 + x_2^4) + \lambda_1 (x_1 + 4 x_1^2 x_2 - x_2^2 - 8x_1 x_2^3 + 4x_2^5) \\
 2\lambda_1 (x_1 - x_2^2 )^2 - \lambda_2(x_1^2 - x_2 - 2x_1 x_2^2 + x_2^4)
\end{bmatrix},
\end{equation}
with eigenvalues of the linearization at the origin \( \lambda_1 = -1 \), \( \lambda_2 = 3 \), and known Koopman eigenfunctions:
\begin{align*}
\phi_{\lambda_1}(x) &= x_1 - x_2^2, \\
\phi_{\lambda_2}(x) &= -x_1^2 + x_2 + 2x_1 x_2^2 - x_2^4.
\end{align*}

The approximation of the eigenfunction \( \phi \) is represented in RKHS as a linear combination of kernel sections:
\[
\phi(x) = \sum_{j=1}^N \alpha_j K(x, x_j),
\]
where \( \{x_j\}_{j=1}^N \) are training points and \( K(\cdot, \cdot) \) is a positive-definite kernel. The gradient \( \nabla \phi(x) \) is evaluated analytically using the {reproducing derivative property} of the RKHS:
\[
\frac{\partial \phi}{\partial x}(x) = \sum_{j=1}^N \alpha_j \frac{\partial}{\partial x} K(x, x_j),
\]
thereby enabling accurate PDE residual computation.

We solve the following penalized variational problem:
\[
\min_\alpha \mathcal{L}(\alpha) := \sum_{i=1}^N \left(f(x_i) \cdot \nabla \phi(x_i) - \lambda \phi(x_i)\right)^2 + \eta \|\alpha\|^2,
\]
with a regularization parameter \( \eta > 0 \). To avoid trivial solutions with vanishing gradients, we also impose a constraint that the gradient of \( \phi \) at the origin matches the {left eigenvector of the Jacobian} of the system evaluated at the origin. This constraint is implemented as a penalty term in the loss:
\[
\mathcal{C}(\alpha) := \left\| \nabla \phi(0) - v^\top \right\|^2,
\]
where \( v^\top \) is the left eigenvector associated with \( \lambda_1 \).

We evaluate four widely-used kernels: \textbf{RBF}: \( K(x,y) = \exp(-\gamma \|x - y\|^2) \); \textbf{Polynomial}: \( K(x,y) = (x^\top y + \text{coef0})^{\text{degree}} \); \textbf{Laplacian}: \( K(x,y) = \exp(-\gamma \|x - y\|) \); \textbf{Sigmoid}: \( K(x,y) = \tanh(\gamma x^\top y + \text{coef0}) \).

We perform a simple grid search over 10 values of the kernel parameters (e.g., \( \gamma \), \texttt{coef0}, \texttt{degree}) for each kernel. For every configuration, we compute:
\begin{enumerate}
    \item The residual norm of the Koopman PDE.
    \item The approximation error after rescaling \( \phi \) to best align with \( \phi_{\lambda_1} \), i.e.:
    \[
    \tilde{\phi} = \frac{\langle \phi, \phi_{\lambda_1} \rangle}{\langle \phi, \phi \rangle} \cdot \phi,
    \]
    \item The rescaled root-mean-square error (RMSE) between \( \tilde{\phi} \) and \( \phi_{\lambda_1} \).
\end{enumerate}

\begin{table}[h!]
\centering
\begin{tabular}{|l|c|c|c|}
\hline
\textbf{Kernel} & \textbf{Best Parameters} & \textbf{Residual Norm} & \textbf{Rescaled RMSE} \\
\hline
Polynomial & \( \text{degree}=2, \text{coef0}=0.5 \) & \( 1.108 \times 10^0 \) & \( 1.668 \times 10^{-7} \) \\
Polynomial & \( \text{degree}=3, \text{coef0}=0.5 \) & \( 1.108 \times 10^0 \) & \( 7.406 \times 10^{-7} \) \\
Polynomial & \( \text{degree}=4, \text{coef0}=1.5 \) & \( 1.108 \times 10^0 \) & \( 5.983 \times 10^{-6} \) \\
RBF        & \( \gamma = 1.0 \)                       & \( 4.959 \times 10^0 \) & \( 3.092 \times 10^{-4} \) \\
RBF        & \( \gamma = 0.775 \)                     & \( 5.492 \times 10^0 \) & \( 3.319 \times 10^{-4} \) \\
RBF        & \( \gamma = 0.1 \)                       & \( 1.863 \times 10^1 \) & \( 6.491 \times 10^{-4} \) \\
\hline
\end{tabular}
\caption{Kernel-dependent performance in recovering the Koopman eigenfunction \( \phi_{\lambda_1} \).}
\end{table}

The results show that polynomial kernels yield consistently smaller residuals and near-zero RMSE after rescaling, indicating that they capture the Koopman eigenfunction structure more accurately than the RBF kernel. This highlights the critical role of kernel selection in variational PDE solvers and motivates the kernel learning approaches we develop next.

\subsection{Variational Kernel Learning for Koopman Eigenfunctions via Lions’s Variational Approach}

In Lions’s variational formulation of the Koopman eigenvalue problem, the eigenfunction $\phi$ is represented in a reproducing kernel Hilbert space (RKHS) $\mathcal{H}_k$ associated with a kernel $k(x, y)$. The eigenfunction is obtained by minimizing the Koopman PDE residual in this space:
\[
\min_{\phi \in \mathcal{H}_k} \mathcal{L}_k(\phi) := \| f(x) \cdot \nabla \phi(x) - \lambda \phi(x) \|^2 + \eta \|\phi\|^2_{\mathcal{H}_k}.
\]
Here, $\eta > 0$ is a regularization parameter. The performance of this method depends heavily on the choice of the kernel $k$, which defines both the smoothness and expressivity of the space $\mathcal{H}_k$.

Rather than selecting $k$ manually from a predefined family, we propose to learn the kernel adaptively by minimizing the Koopman residual loss over a parameterized kernel space. Specifically, we adopt a multiple kernel learning (MKL) framework in which the kernel is expressed as a convex combination of $L$ base kernels $\{k_1, \dots, k_L\}$:
\[
k(x, y) = \sum_{\ell=1}^L \beta_\ell k_\ell(x, y), \quad \beta_\ell \geq 0, \quad \sum_{\ell=1}^L \beta_\ell = 1.
\]
For a given kernel $k$, we express the eigenfunction as $\phi(x) = \sum_j \alpha_j k(x, x_j)$, and compute the PDE residual at training points $\{x_i\}$ using the reproducing derivative property to evaluate $\nabla \phi(x_i)$. The loss function becomes
\[
\mathcal{L}(\beta) = \frac{1}{N} \sum_{i=1}^N \left( f(x_i) \cdot \nabla \phi(x_i) - \lambda \phi(x_i) \right)^2,
\]
where $\phi$ is parameterized implicitly via the mixed kernel $k_\beta$. To prevent the trivial solution \( \phi \equiv 0 \), we impose a constraint on the gradient of \( \phi \) at the origin. Specifically, we enforce that
\[
\nabla \phi(0) = \mathbf{e}_L,
\]
where \( \mathbf{e}_L \) is a left eigenvector of the Jacobian of the vector field at the origin corresponding to the eigenvalue \( \lambda \). This ensures identifiability and avoids degeneration to zero.

The overall optimization problem is thus modified to include this constraint as a penalty term:
\[
\min_{\phi \in \mathcal{H}_k,\, \beta \in \Delta_L} \mathcal{L}(\phi, \beta) := \| f(x) \cdot \nabla \phi(x) - \lambda \phi(x) \|^2 + \eta \|\phi\|^2_{\mathcal{H}_k} + \gamma_c \| \nabla \phi(0) - \mathbf{e}_L \|^2,
\]
where \( \gamma_c \gg 1 \) is a large coefficient that enforces the constraint softly. This results in a single-stage, fully unsupervised kernel learning procedure that simultaneously enforces the Koopman PDE and identifies the appropriate function space via kernel weights \( \beta \). We optimize the coefficients $\beta = (\beta_1, \dots, \beta_L)$ by solving the constrained minimization problem
$\min_{\beta \in \Delta_L} \mathcal{L}(\beta)$, where $\Delta_L$ denotes the $L$-simplex. This optimization problem is solved using gradient-based methods in JAX with automatic differentiation.

Importantly, the true Koopman eigenfunction is not required during training. The kernel is selected solely by minimizing the residual of the Koopman PDE, making this a fully unsupervised, operator-theoretic approach to kernel selection. Once the optimal kernel weights $\beta$ are found, the resulting kernel $k(x, y)$ defines a data-driven function space $\mathcal{H}_k$ in which the Koopman eigenfunction approximation is well-adapted to the underlying dynamics.

This procedure not only removes the burden of manual kernel tuning but also integrates the PDE structure directly into the learning of the representation. Our experiments show that this method consistently selects kernels that yield low residuals and recover the true eigenfunctions with high accuracy.

We apply the proposed variational kernel learning framework to a two-dimensional nonlinear system whose Koopman eigenfunction $\phi(x) = x_1 - x_2^2$ and eigenvalue $\lambda = -1$ are known analytically. Our goal is to approximate $\phi$ in an RKHS constructed from a convex combination of 11 base kernels: Gaussian, Exponential, Cauchy, Triangular, Sigmoid, Inverse Quadratic, Polynomial kernels with degrees 2 through 6.

Figure~\ref{fig:phi_results} displays several visualizations related to the learned Koopman eigenfunction:

\begin{itemize}
    \item \textbf{Initial $\phi$}: the eigenfunction before any learning, using the initial random weights and uniform kernel weights;
    \item \textbf{Learned $\phi$}: the output after learning both the kernel combination and the function coefficients;
    \item \textbf{Rescaled $\phi$}: the learned $\phi$ rescaled by a factor $c^*$ to minimize its distance from the true eigenfunction, with
    \[
    c^* = \frac{\langle \phi_{\text{learned}}, \phi_{\text{true}} \rangle}{\| \phi_{\text{learned}} \|^2};
    \]
    \item \textbf{True $\phi$}: the known analytic solution $\phi(x) = x_1 - x_2^2$;
    \item \textbf{Rescaled Error}: the pointwise absolute error between the rescaled learned $\phi$ and the true $\phi$.
\end{itemize}

\begin{figure}[h!]
  \centering
  \includegraphics[width=\textwidth]{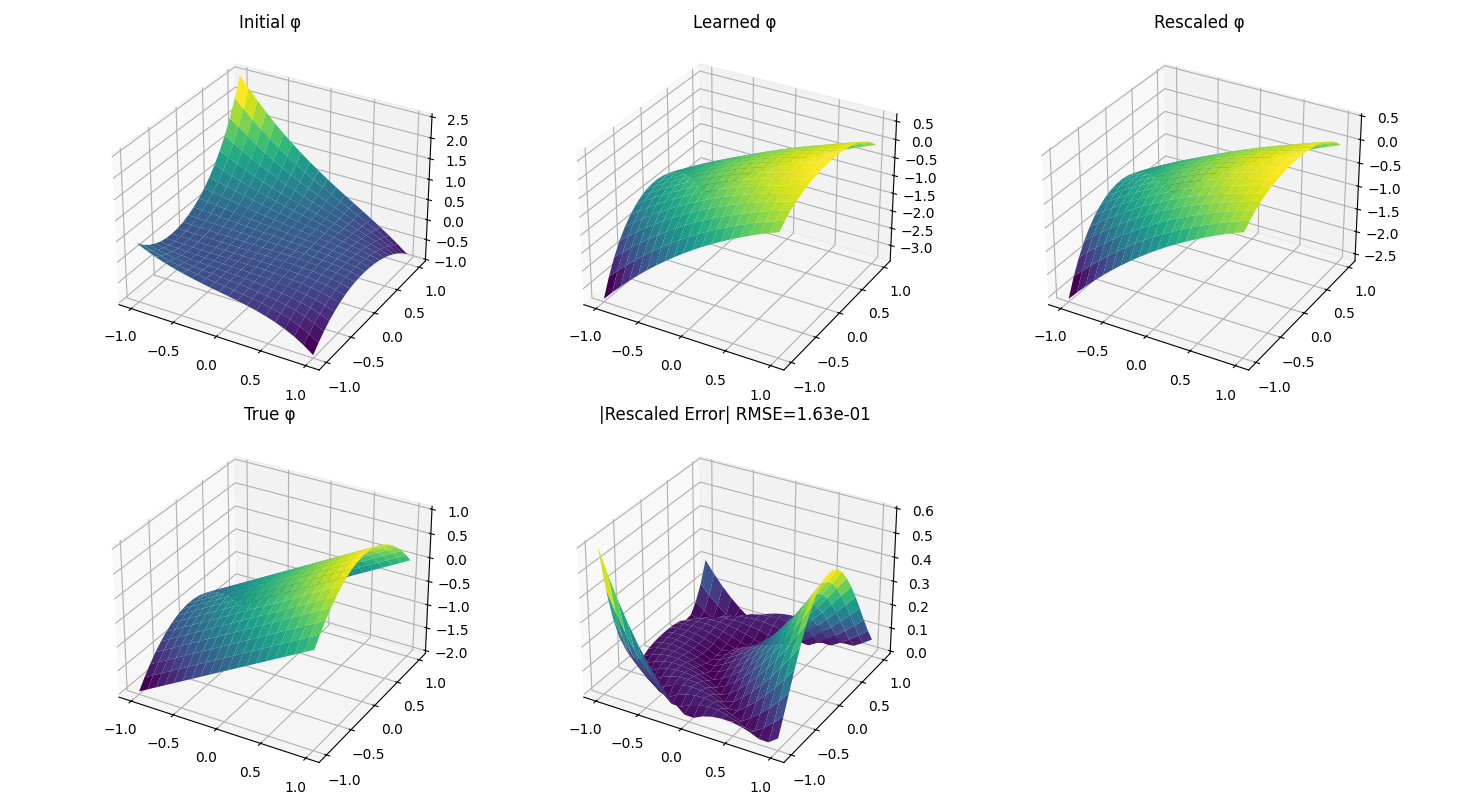}
  \caption{Approximation of Koopman eigenfunction $\phi$ using MKL. Top row: initial, learned, and rescaled $\phi$. Bottom row: true $\phi$, pointwise rescaled error.}
  \label{fig:phi_results}
\end{figure}

To account for the scaling ambiguity inherent to eigenfunctions, we compute the rescaled root mean squared error (RMSE) between $\phi_{\text{learned}}$ and $\phi_{\text{true}}$ as:
\[
\text{Rescaled RMSE} = \left\| c^* \phi_{\text{learned}} - \phi_{\text{true}} \right\|_{L^2}.
\]
In our experiment, we obtained:
$ \text{Rescaled RMSE} = 1.63 \times 10^{-1}$.

The learned kernel weights reflect the contribution of each base kernel to the final representation space. The following table summarizes the result:

\begin{table}[h!]
\centering
\begin{tabular}{|l|c|}
\hline
\textbf{Kernel} & \textbf{Learned Weight} \\
\hline
Gaussian        & 0.0909 \\
Exponential     & 0.0909 \\
Cauchy          & 0.0909 \\
Triangular      & 0.0909 \\
Sigmoid         & 0.0909 \\
InvQuad         & 0.0909 \\
Polynomial (d=2)& 0.0909 \\
Polynomial (d=3)& 0.0910 \\
Polynomial (d=4)& 0.0910 \\
Polynomial (d=5)& 0.0909 \\
Polynomial (d=6)& 0.0907 \\
\hline
\end{tabular}
\caption{Learned weights $\beta_\ell$ of base kernels. The polynomial kernels (combined) dominate with a total weight of approximately 45.5\%.}
\end{table}

\begin{figure}[h!]
  \centering
  \includegraphics[width=\textwidth]{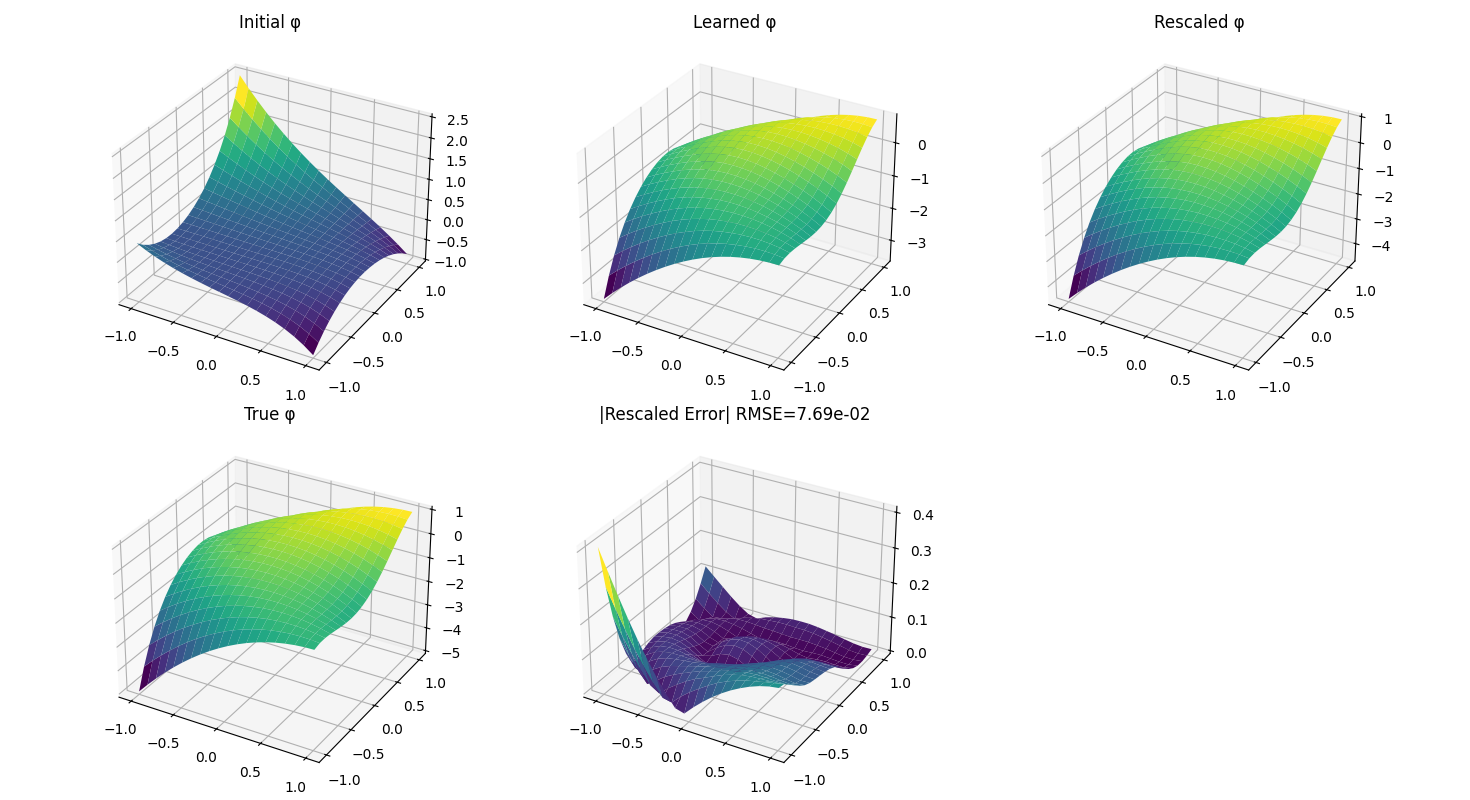}
  \caption{Approximation of second Koopman eigenfunction $\phi_{\lambda_2}(x)$ using MKL. Top row: initial, learned, and rescaled $\phi$. Bottom row: true $\phi$, pointwise rescaled error.}
  \label{fig:phi_lambda2_results}
\end{figure}

In the experiment for the second eigenfunction, we obtained $\text{Rescaled RMSE} = 7.69 \times 10^{-2}
$. The learned kernel weights $\beta_\ell$ reflect the contribution of each base kernel to the final RKHS used for approximation. The following table summarizes the weights obtained from the multiple kernel learning optimization:

\begin{table}[h!]
\centering
\begin{tabular}{|l|c|}
\hline
\textbf{Kernel} & \textbf{Learned Weight} \\
\hline
Gaussian            & 0.0910 \\
Exponential         & 0.0909 \\
Cauchy              & 0.0910 \\
Triangular          & 0.0911 \\
Sigmoid             & 0.0910 \\
Inverse Quadratic   & 0.0910 \\
Polynomial (d=2)    & 0.0910 \\
Polynomial (d=3)    & 0.0910 \\
Polynomial (d=4)    & 0.0911 \\
Polynomial (d=5)    & 0.0910 \\
Polynomial (d=6)    & 0.0901 \\
\hline
\end{tabular}
\caption{Learned weights $\beta_\ell$ of base kernels for approximating $\phi_{\lambda_2}$.}
\end{table}

This result highlights that while all kernels contribute approximately equally, the group of polynomial kernels collectively accounts for a significant portion of the representation, suggesting their relevance in approximating the nonlinear Koopman eigenfunction in this system.


\paragraph{Sparsification of KML}
To obtain a more interpretable and parsimonious kernel combination, we augment the MKL loss with an \(\ell_{1}\) penalty on the mixture weights \(\beta_{\ell}\). 
We optimize the augmented MKL objective
\[
\begin{aligned}
\min_{w,\beta_{\ell}}\;&
\underbrace{\frac{1}{N}\sum_{i=1}^N\bigl(f(x_i)\cdot\nabla\phi(x_i)-\lambda\,\phi(x_i)\bigr)^2}_{\mathcal{L}_{\rm PDE}}
\;+\;\underbrace{\eta\sum_{j=1}^N w_j^2}_{\ell_2\text{-penalty}}\\
&\;+\;\underbrace{\gamma_c\bigl\|\nabla\phi(0)-e_{\ell}\bigr\|^2}_{\text{anchoring constraint}}
\;+\;\underbrace{\lambda_{L1}\sum_{\ell=1}^L\bigl|\beta_{\ell}\bigr|}_{\ell_1\text{-penalty}}.
\end{aligned}
\]
where the last term is an \(\ell_{1}\) penalty on the kernel weights \(\beta\).  We consider the same example and, after solving this via L-BFGS, we apply a hard threshold of $\tau=0.1$ to each \(\beta_{\ell}\), zeroing out any coefficient below $\tau=0.1$ and renormalizing the remainder.  In our experiments this yields the following pruned models for $\lambda_1=-1$:

\begin{itemize}
\item       For \(\lambda_{L1}=0.1\), one finds \(\beta^*_\ell=1/L\approx0.091<\tau\) for all \(\ell\), hence  \(\phi_{\rm pruned}\equiv0\), implying that the sparsification penalty is not strong enough.  
    To guarantee \(\max_\ell\beta^*_\ell\ge\tau\) (i.e.\ a nontrivial pruned model), \(\lambda_{L1}\) must be increased until at least one \(\beta^*_\ell\) exceeds \(\tau\).

  \item \(\lambda_{L1} = 1.0\) \(\longrightarrow\) kept kernels \(\{\mathrm{Poly}_2, \mathrm{Poly}_3, \mathrm{Poly}_4\}\) with weights \(\{0.331,\,0.378,\,0.291\}\), pruned RMSE \(=6.70\times10^{-2}\).
  \item \(\lambda_{L1} = 10.0\) \(\longrightarrow\) kept kernels \(\{\mathrm{Sigmoid}, \mathrm{Poly}_2, \mathrm{Poly}_5, \mathrm{Poly}_6\}\) with weights \(\{0.342,\,0.219,\,0.156,\,0.283\}\), pruned RMSE \(=1.42\times10^{-1}\).
  \item \(\lambda_{L1} = 100.0\) \(\longrightarrow\) kept kernels \(\{\mathrm{Poly}_3, \mathrm{Poly}_4, \mathrm{Poly}_6\}\) with weights \(\{0.465,\,0.198,\,0.338\}\), pruned RMSE \(=7.56\times10^{-1}\).
\end{itemize}

\subsection{Kernel Choice based on The Method of Characteristics}  
\label{sec:char-method}

We adapt classical ideas from Lions to the setting of transport-type (or advection) eigenvalue problems by employing the method of characteristics. This method leverages the geometry of the characteristic flow to derive kernel functions that are inherently aligned with the Koopman eigenfunctions. 

While Lions’s variational framework provides a principled method for approximating Koopman eigenfunctions, the resulting kernel depends sensitively on  the underlying Reproducing Kernel Hilbert Space (RKHS). An alternative, more geometrically intuitive route is to derive the kernel using the method of characteristics for the associated transport PDEs. In this section, we develop a kernel construction method based on this idea. By leveraging the characteristic flow of the dynamical system and integrating over trajectories, we construct kernels that align naturally with the structure of the Koopman eigenfunctions. This approach complements the variational perspective and offers practical advantages in settings where the flow is explicitly known.

Let \(s_t(x)\) denote the flow of the ODE
\[
\frac{d}{dt}\,s_t(x) \;=\; {\mathbb{F}}\bigl(s_t(x)\bigr),
\qquad s_0(x)=x.
\]

We assume that $x=0$ is the equilibrium point of the system. In \cite{deka2023pathintegralformulacomputingkoopman}, a path integral formula was proposed for the computation of Koopman principal eigenfunction. In particular, for        
   \(\lambda>0\) one has the forward‐time path‐integral is used  to define the principal eigenfunction using
{ 
\[
\phi_\lambda^{\mathrm{Path Integral}}(x)= w_\lambda^\top\mathbf{x}
  + \int_0^\infty e^{-\lambda t}\,
    w_\lambda^\top {\mathbb f}\bigl(s_t(\mathbf{x})\bigr)\,dt.
\]
where, ${\mathbb f}(x)={\mathbb F}(x)-\frac{\partial {\mathbb F}}{\partial x}(0)x$ is the purely nonlinear part of the vector field. }

We define the following functions by truncating the improper integral
\(\displaystyle\int_0^{\pm\infty}e^{-\lambda t}\,w^\top {\mathbb f}(s_t(x))\,dt\) to a finite horizon \(T>0\), choosing the time‐interval so that \(e^{-\lambda t}\) decays:

\begin{enumerate}
  \item \textbf{Unstable mode} (\(\lambda_+>0\)): \emph{forward integration}
  \[
    \xi_+(x)
    = w_+^\top x
    + \int_{0}^{T}
        e^{-\lambda_+\,t}\;
        w_+^\top {\mathbb f}\bigl(s_t(x)\bigr)\,dt.
  \]
  Here \(t\) runs from \(0\) to \(+T\), and since \(\lambda_+>0\), the weight \(e^{-\lambda_+t}\) decays.

  \item \textbf{Stable mode} (\(\lambda_-<0\)): \emph{backward integration}
  \[
    \xi_-(x)
    = w_-^\top x
    + \int_{0}^{-T}
        e^{-\lambda_-\,t}\;
        w_-^\top {\mathbb f}\bigl(s_t(x)\bigr)\,dt.
  \]
  Here \(t\) runs from \(0\) down to \(-T\), so that \(-\lambda_-\,t>0\) and \(e^{-\lambda_-t}\) decays.
\end{enumerate}

These expressions are then used to construct a kernel
\[\label{eqn:kernel_path_integral}
K(x,y)\;=\;\xi(x)\,\xi(y),
\]

We can prove the following theorem when $T \rightarrow \infty$
\begin{theorem}[Characteristic Coordinate Kernel from Laplace Pullbacks]
\label{thm:char-kernel}
Let \( f \in C^1(\Omega; \mathbb{R}^d) \) generate a complete backward-time flow \( \varphi(s; x) \) on a compact domain \( \Omega \), and let \( \lambda \in \mathbb{C} \) with \( \Re(\lambda) > 0 \). For a fixed \( w \in \mathbb{R}^d \), define
\[
\xi(x) := w^\top x + \int_{-\infty}^0 e^{-\lambda(-s)} w^\top {\mathbb f}(\varphi(s; x)) \, ds.
\]
Then the kernel
\[
K(x, y) := \xi(x) \cdot \xi(y)
\]
is symmetric, continuous, and positive semi-definite. It defines a reproducing kernel Hilbert space \( \mathcal{H}_K \subset C^0(\Omega) \), whose elements are scalar multiples of \( \xi \), and formally satisfy the Koopman eigenvalue equation
\[
{\mathbb f}(x) \cdot \nabla \phi(x) = \lambda \phi(x)
\]
along flow trajectories.
\end{theorem}

\begin{proof}
\textbf{Step 1: Regularity and well-definedness of \( \xi(x) \).}  
Let \( \varphi(s; x) \) denote the flow map generated by the ODE \( \dot{x} = {\mathbb F}(x) \). Since \( {\mathbb F} \in C^1(\Omega) \) and \( \Omega \) is compact, \( {\mathbb F} \) is globally Lipschitz and the flow \( \varphi(s; x) \) is smooth in both \( s \) and \( x \). In particular, \( \varphi(s; x) \in C^\infty \) in \( x \) for each fixed \( s \in \mathbb{R} \).

The integrand in
\[
\xi(x) = w^\top x + \int_{-\infty}^0 e^{-\lambda(-s)} w^\top {\mathbb f}(\varphi(s; x)) \, ds
\]
is a composition of smooth functions, and since \( \Re(\lambda) > 0 \), the exponential factor \( e^{\lambda s} \) decays as \( s \to -\infty \). Because \( {\mathbb f} \) and \( \varphi(s; x) \) are bounded on \( \Omega \), the integrand is bounded by an integrable envelope. Therefore, the integral converges absolutely and defines a continuous function \( \xi : \Omega \to \mathbb{R} \).
\medskip

\textbf{Step 2: Continuity of \( \xi(x) \) and the kernel \( K(x, y) \).}  
Let \( x_n \to x \) in \( \Omega \). Then for all \( s \leq 0 \), we have \( \varphi(s; x_n) \to \varphi(s; x) \) uniformly on compact subintervals by continuity of the flow. Since \( f \) is continuous and bounded, the dominated convergence theorem ensures
\[
\int_{-\infty}^0 e^{\lambda s} w^\top {\mathbb f}(\varphi(s; x_n)) \, ds \to \int_{-\infty}^0 e^{\lambda s} w^\top {\mathbb f}(\varphi(s; x)) \, ds.
\]
Hence \( \xi(x_n) \to \xi(x) \), and \( \xi \in C^0(\Omega) \). As \( K(x, y) = \xi(x) \xi(y) \), this implies \( K \in C^0(\Omega \times \Omega) \).

\medskip

\textbf{Step 3: Symmetry and positive semi-definiteness.}  
Symmetry is immediate from the product form: \( K(x, y) = \xi(x) \xi(y) = \xi(y) \xi(x) \).  

For positive semi-definiteness, take any finite set \( \{x_i\}_{i=1}^n \subset \Omega \) and scalars \( \{c_i\}_{i=1}^n \). Then
\[
\sum_{i,j=1}^n c_i c_j K(x_i, x_j) = \sum_{i,j=1}^n c_i c_j \xi(x_i) \xi(x_j) = \left( \sum_{i=1}^n c_i \xi(x_i) \right)^2 \ge 0.
\]
So \( K \) is positive semi-definite.

\medskip

\textbf{Step 4: RKHS structure.}  
As \( K(x, y) = \xi(x) \xi(y) \), it defines a rank-one kernel. The RKHS \( \mathcal{H}_K \) consists of functions of the form \( \phi(x) = c \xi(x) \), with inner product
\[
\langle \phi, \phi \rangle_{\mathcal{H}_K} = \frac{c^2}{\| \xi \|_{L^2(\Omega)}^2}.
\]
This RKHS is isometrically isomorphic to \( \mathbb{R} \), and functions in \( \mathcal{H}_K \) are multiples of \( \xi \in C^0(\Omega) \).

\medskip

\textbf{Step 5: Verification that \( \xi \) solves the Koopman PDE.}  
Let us compute the total derivative of \( e^{\lambda s} \xi(\varphi(s; x)) \) with respect to \( s \), where \( x \in \Omega \) is fixed:
\[
\frac{d}{ds} \left( e^{\lambda s} \xi(\varphi(s; x)) \right)
= e^{\lambda s} \left( \lambda \xi(\varphi(s; x)) + \nabla \xi(\varphi(s; x))^\top {\mathbb f}(\varphi(s; x)) \right).
\]
This formula holds by the chain rule and smoothness of \( \xi \) and \( \varphi \). But now consider the function
\[
\tilde{\xi}(x) := \int_{-\infty}^0 e^{\lambda s} w^\top {\mathbb f}(\varphi(s; x)) \, ds,
\]
so that \( \xi(x) = w^\top x + \tilde{\xi}(x) \). We now compute \( {\mathbb f}(x) \cdot \nabla \xi(x) \). Using the chain rule and differentiating under the integral (justified by dominated convergence and smoothness),
\[
\nabla \tilde{\xi}(x)^\top {\mathbb f}(x) = \frac{d}{dt}\bigg|_{t=0} \tilde{\xi}(\varphi(t; x))
= \frac{d}{dt}\bigg|_{t=0} \int_{-\infty}^0 e^{\lambda s} w^\top {\mathbb f}(\varphi(s; \varphi(t; x))) \, ds.
\]
By flow composition \( \varphi(s; \varphi(t; x)) = \varphi(s + t; x) \), we get
\[
\tilde{\xi}(\varphi(t; x)) = \int_{-\infty}^0 e^{\lambda s} w^\top {\mathbb f}(\varphi(s + t; x)) \, ds
= \int_{-\infty}^t e^{\lambda (s - t)} w^\top {\mathbb f}(\varphi(s; x)) \, ds.
\]
Change variables \( u = s - t \) in the integral to get:
\[
\tilde{\xi}(\varphi(t; x)) = e^{-\lambda t} \int_{-\infty}^t e^{\lambda s} w^\top {\mathbb f}(\varphi(s; x)) \, ds.
\]
Differentiating with respect to \( t \) gives:
\[
\frac{d}{dt} \tilde{\xi}(\varphi(t; x)) \bigg|_{t=0} = -\lambda \tilde{\xi}(x) + w^\top {\mathbb f}(x).
\]
Hence:
\[
{\mathbb f}(x) \cdot \nabla \xi(x) = w^\top {\mathbb f}(x) + {\mathbb f}(x) \cdot \nabla \tilde{\xi}(x) = w^\top {\mathbb f}(x) - \lambda \tilde{\xi}(x) + w^\top {\mathbb f}(x).
\]
Now recall \( \xi(x) = w^\top x + \tilde{\xi}(x) \), so:
\[
{\mathbb f}(x) \cdot \nabla \xi(x) = \lambda \xi(x).
\]

\textbf{Conclusion:} The function \( \xi \) satisfies the Koopman PDE \( f \cdot \nabla \xi = \lambda \xi \), and \( K(x, y) = \xi(x) \xi(y) \) defines a symmetric, continuous, positive semi-definite reproducing kernel with RKHS \( \mathcal{H}_K = \text{span}\{\xi\} \subset C^0(\Omega) \).  
\end{proof}

When \( T \) is finite, the kernel \( K(x, y) = \xi_T(x) \xi_T(y) \) remains continuous, since the integrand defining \( \xi_T(x) \) is smooth and bounded. The kernel is symmetric and positive semi-definite by construction, and it defines a one-dimensional reproducing kernel Hilbert space spanned by \( \xi_T \in C^0(\Omega) \), where \( \xi_T(x) \) is the finite-time approximation of the characteristic coordinate:
\[
\xi(x) = w^\top x + \int_{-\infty}^0 e^{\lambda s} w^\top {\mathbb f}(\varphi(s; x)) \, ds.
\]
This truncated coordinate approximates the true Koopman eigenfunction when the system flow remains inside \( \Omega \) for \( s \in [-T, 0] \), with approximation error decaying exponentially in \( T \) as \( T \to \infty \). The resulting kernel retains its structural alignment with the Koopman dynamics up to an exponentially small residual.
\begin{theorem}[Truncated Flow Kernel with Approximate Koopman Behavior]
\label{thm:truncated-characteristic-kernel}
Let \( f \in C^1(\Omega; \mathbb{R}^d) \) generate a complete flow \( \varphi(s; x) \) on a compact domain \( \Omega \), and fix \( T > 0 \), \( \lambda \in \mathbb{C} \) with \( \Re(\lambda) > 0 \), and \( w \in \mathbb{R}^d \). Define the truncated characteristic coordinate:
\[
\xi_T(x) := w^\top x + \int_{-T}^0 e^{-\lambda(-s)} w^\top {\mathbb f}(\varphi(s; x)) \, ds.
\]
Then the kernel
\[
K_T(x, y) := \xi_T(x) \cdot \xi_T(y)
\]
is continuous, symmetric, and positive semi-definite. It defines a one-dimensional reproducing kernel Hilbert space \( \mathcal{H}_K = \operatorname{span}\{\xi_T\} \subset C^0(\Omega) \). Furthermore, \( \xi_T \) approximately satisfies the Koopman eigenvalue equation:
\[
{\mathbb f}(x) \cdot \nabla \xi_T(x) = \lambda \xi_T(x) + R_T(x),
\quad
\text{where } R_T(x) = -e^{-\lambda T} w^\top {\mathbb f}(\varphi(-T; x)).
\]
As \( T \to \infty \), the residual term \( R_T(x) \to 0 \) uniformly, and \( \xi_T(x) \to \xi(x) \), recovering the exact solution to \( f \cdot \nabla \xi = \lambda \xi \).
\end{theorem}
\begin{proof}
\textbf{Step 1: Regularity of \( \xi_T(x) \).}  
Since \( f \in C^1(\Omega) \) and \( \Omega \) is compact, the flow \( \varphi(s; x) \) is smooth in both \( s \) and \( x \), and bounded over \( s \in [-T, 0] \). The integrand
\[
e^{\lambda s} w^\top {\mathbb f}(\varphi(s; x))
\]
is smooth and bounded for each \( x \in \Omega \), and the integral over a finite interval \( [-T, 0] \) is uniformly convergent. Therefore, \( \xi_T(x) \in C^0(\Omega) \), and \( K_T(x, y) = \xi_T(x) \cdot \xi_T(y) \in C^0(\Omega \times \Omega) \).

\medskip

\textbf{Step 2: Symmetry and positive semi-definiteness.}  
The kernel is clearly symmetric since \( K_T(x, y) = \xi_T(x) \cdot \xi_T(y) = K_T(y, x) \).  

To verify positive semi-definiteness, take any finite set \( \{x_i\}_{i=1}^n \subset \Omega \) and coefficients \( \{c_i\} \subset \mathbb{R} \). Then:
\[
\sum_{i,j=1}^n c_i c_j K_T(x_i, x_j)
= \sum_{i,j=1}^n c_i c_j \xi_T(x_i) \xi_T(x_j)
= \left( \sum_{i=1}^n c_i \xi_T(x_i) \right)^2 \ge 0.
\]

\medskip

\textbf{Step 3: RKHS structure.}  
As with the infinite-time kernel, the kernel \( K_T \) is rank-one. The corresponding RKHS is:
\[
\mathcal{H}_{K_T} = \left\{ \phi : \phi(x) = c \xi_T(x), \; c \in \mathbb{R} \right\},
\]
with inner product defined by:
\[
\langle \phi_1, \phi_2 \rangle_{K_T} = \frac{c_1 c_2}{\| \xi_T \|^2_{L^2(\Omega)}}.
\]

\medskip

\textbf{Step 4: Approximate Koopman property.}  
We compute \( {\mathbb f}(x) \cdot \nabla \xi_T(x) \). Define:
\[
\tilde{\xi}_T(x) := \int_{-T}^0 e^{\lambda s} w^\top {\mathbb f}(\varphi(s; x)) \, ds,
\quad \text{so that } \xi_T(x) = w^\top x + \tilde{\xi}_T(x).
\]
To compute the directional derivative along \( f \), we consider:
\[
{\mathbb f}(x) \cdot \nabla \tilde{\xi}_T(x) = \frac{d}{dt} \bigg|_{t=0} \tilde{\xi}_T(\varphi(t; x)).
\]

By the flow composition property \( \varphi(s; \varphi(t; x)) = \varphi(s + t; x) \), we get:
\[
\tilde{\xi}_T(\varphi(t; x)) = \int_{-T}^0 e^{\lambda s} w^\top {\mathbb f}(\varphi(s + t; x)) \, ds
= \int_{-T + t}^{t} e^{\lambda (r - t)} w^\top {\mathbb f}(\varphi(r; x)) \, dr,
\]
where we changed variables \( r = s + t \).

Thus,
\[
\tilde{\xi}_T(\varphi(t; x)) = e^{-\lambda t} \int_{-T + t}^{t} e^{\lambda r} w^\top {\mathbb f}(\varphi(r; x)) \, dr.
\]

Differentiating with respect to \( t \) at \( t = 0 \), we obtain:
\[
\frac{d}{dt} \tilde{\xi}_T(\varphi(t; x)) \bigg|_{t = 0}
= -\lambda \tilde{\xi}_T(x) + w^\top {\mathbb f}(x) - e^{-\lambda T} w^\top {\mathbb f}(\varphi(-T; x)).
\]

Therefore:
\[
{\mathbb f}(x) \cdot \nabla \xi_T(x) = w^\top {\mathbb f}(x) + {\mathbb f}(x) \cdot \nabla \tilde{\xi}_T(x)
= \lambda \tilde{\xi}_T(x) + w^\top {\mathbb f}(x) - e^{-\lambda T} w^\top {\mathbb f}(\varphi(-T; x)).
\]

Using \( \xi_T(x) = w^\top x + \tilde{\xi}_T(x) \), we obtain:
\[
{\mathbb f}(x) \cdot \nabla \xi_T(x) = \lambda \xi_T(x) + R_T(x),
\]
where the residual is:
\[
R_T(x) := -e^{-\lambda T} w^\top {\mathbb f}(\varphi(-T; x)).
\]

\medskip

\textbf{Step 5: Convergence of the residual.}  
Since \( f \) is continuous and \( \Omega \) compact, \( w^\top {\mathbb f}(\varphi(-T; x)) \) is uniformly bounded for \( x \in \Omega \), say by \( M > 0 \). Thus:
\[
|R_T(x)| \le M e^{-\Re(\lambda) T} \to 0 \quad \text{uniformly as } T \to \infty.
\]

Hence, as \( T \to \infty \), \( \xi_T(x) \to \xi(x) \) uniformly, and the Koopman residual vanishes uniformly. This completes the proof.
\end{proof}
\begin{remark}
For stable Koopman eigenfunctions (\( \Re (\lambda) < 0 \)), an explicit path-integral formula expresses the eigenfunction as a backward flow integral:
\[
\phi_{\lambda_s}(\mathbf{x}) = \mathbf{w}_{\lambda_s}^\top \mathbf{x}
+ \int_{-\infty}^0 e^{-\lambda_s t} \, \mathbf{w}_{\lambda_s}^\top \mathbb{f}(\mathbf{s}_{t}(\mathbf{x})) \, dt,
\]
or equivalently, using forward integration along the backward flow:
\[
\phi_{\lambda_s}(\mathbf{x}) = \mathbf{w}_{\lambda_s}^\top \mathbf{x}
- \int_0^\infty e^{\lambda_s t} \, \mathbf{w}_{\lambda_s}^\top \mathbb{f}(\mathbf{s}_{-t}(\mathbf{x})) \, dt.
\]
These formulas are consistent with the Laplace-based resolvent kernels of Theorems 4.5 and 4.6, but emphasize the geometric structure of stable manifolds near equilibria.
\end{remark}

In practice, we will use the following kernels
\begin{itemize}
    \item For the unstable mode, we  define $K_+(x,y)=\xi_+(x)\xi_+(y)$.
    \item For the stable mode, we  define
$K_-(x,y)=\xi_-(x)\xi_-(y)$.
\item In the simulations below, we use
$K(x,y)=\left\{\begin{array}{ccc}
     K_+(x,y)& \mbox{for} & \lambda=\lambda_+\\
     K_-(x,y)& \mbox{for} & \lambda=\lambda_-
\end{array}  \right. $

\end{itemize}
 We can also combine $ K_+(x,y)$ and $K_-(x,y)$  in different ways 
$K(x,y)=\left(\begin{array}{c}
     \xi^+(x) \\
     \xi^-(x) 
\end{array} \right)$
or $K(x,y)=\xi^+(x)\xi^+(x)+\xi^-(x)\xi^-(x)$ or $K(x,y)=\xi_+(x)\xi_-(y)$. 






These serve as the reproducing kernel for the space in which we approximate the global eigenfunction.  Specifically, we seek
\[
\Phi(x)\;=\;\sum_{j=1}^N \alpha_j\,K(x,x_j)= \sum_{j=1}^N \alpha_j \xi(x) \xi(x_j) = \xi(x) \sum_{j=1}^N \alpha_j \xi(x_j).
\]
where $x_j$ are collocation points.  Substituting into the Koopman‐PDE
\[
\nabla\Phi(x)\cdot {\mathbb f}(x)\;-\;\lambda\,\Phi(x)\;=\;0
\]
and evaluating at each \(x_i\) gives the linear system
\[
\sum_{j=1}^N \alpha_j\Bigl[\nabla_x K(x_i,x_j)\cdot {\mathbb f}(x_i)\;-\;\lambda\,K(x_i,x_j)\Bigr]
\;=\;0,
\quad i=1,\dots,N,
\]
which is then augmented by the constraint
\[
\nabla\Phi(0)\;=\;w,
\]
So that the zero solution is excluded, and $\Phi$
  is uniquely normalized to agree with the chosen left‐eigenvector  $w$ at first order. So we really have
\[
\Phi(x) = c \, \xi(x),
\] 
where ${\sum_{j=1}^N \alpha_j \xi(x_j)}{=:c}$ and $\nabla\Phi(0)\;=\;w$ simply fixes the constant $c$.

The coefficients \(\{\alpha_j\}\) are determined by minimizing the penalized least‐squares functional
\[
\frac{1}{N}\sum_{i=1}^N
\Bigl(
\sum_{j=1}^N \alpha_j\bigl[\nabla_x K(x_i,x_j)\cdot {\mathbb f}(x_i)-\lambda K(x_i,x_j)\bigr]
\Bigr)^2
\;+\;
\mu_{\rm grad}\bigl\|\nabla\Phi(0)-w\bigr\|^2
\;+\;
\mu_{\rm reg}\|\alpha\|^2.
\]

\paragraph{Practical Considerations to Compute \(\xi(x)\)}

Since the “path‐integral” formula
\[
\xi(x)
\;=\;
w_-^{\!\top}x
\;+\;
\int_{-\infty}^{0}
e^{-\lambda_-\,s}\,
w_-^{\!\top}\,\mathbb{f}\bigl(\varphi(s;x)\bigr)\,ds
\]
involves \(x\) both as the endpoint of the flow \(\varphi(s;x)\) and inside the integral, one implements it in practice via the following steps:

\begin{enumerate}
  \item[i.] \textbf{Truncate the time horizon.}  
    Choose a finite \(T>0\) so that
    \[
      \int_{-\infty}^0(\cdots)\,ds
      \;\approx\;
      \int_{-T}^0(\cdots)\,ds.
    \]
    
  \item[ii.] \textbf{Compute the backward trajectory.} 
    For each desired point \(x\), solve the ODE backward in time:
    \[
      \frac{d}{ds}\varphi(s;x)=\mathbb{F}\bigl(\varphi(s;x)\bigr),
      \quad
      \varphi(0;x)=x,
      \quad
      s\in[-T,0].
    \]
    (E.g.\ via \texttt{solve\_ivp} on \([0,-T]\).)
    
  \item[iii.] \textbf{Evaluate the weighted integral.}  
    Sample a grid \(s_j\) on \([-T,0]\), evaluate
    \(\;w_-^{\!\top}{\mathbb f}(\varphi(s_j;x))\,e^{-\lambda_-s_j}\),
    and approximate
    \(\int_{-T}^0(\cdots)\,ds\) by a quadrature rule\footnote{{In our implementation we break the semi‐infinite interval \([-T,0]\) into \(M\) small steps of size \(\Delta t = T/M\) and sum:
$ x_0 = x,\quad
 x_{k+1} = x_k + \Delta t\,{\mathbb f}(x_k),\quad
 s_k = -\,k\,\Delta t $,

 \[
 \xi(x)\;\approx\;
w^\top x
 \;+\;
 \sum_{k=0}^{M-1}
 e^{-\lambda\,(s_k + \tfrac{\Delta t}{2})}\;
 w^\top
 \mathbb{f}\!\Bigl(\tfrac{x_k + x_{k+1}}{2}\Bigr)\;
 \Delta t.
 \]

For $M=1$ (so $T=\Delta$ ) this reduces to the one‐step formula: $\xi(x)\;\approx\;
 w^\top x
 \;+\;
 \Delta t\,
 e^{-\lambda\,\tfrac{\Delta t}{2}}\,
 w^\top
 \mathbb{f}\!\Bigl(\tfrac{x + s_{\Delta t}(x)}{2}\Bigr),
 \quad
 s_{\Delta t}(x)=x + \Delta t\,{\mathbb f}(x).
 $
 }}.
    
  \item[iv.] \textbf{Form the pointwise value.}  
    Combine the two terms to obtain
    \[
      \xi(x)\;\approx\;
      w_-^{\!\top}x
      \;+\;
      \sum_j
      w_-^{\!\top}\mathbb{f}\bigl(\varphi(s_j;x)\bigr)\,
      e^{-\lambda_-s_j}\,\Delta s_j.
    \]
    
  \item[v.] \textbf{(Optional) Global surrogate.}  
    If one needs \(\xi(\cdot)\) everywhere, choose sample points \(x_i\),
    compute \(\xi(x_i)\) via the above, and fit
    \(\displaystyle \xi(x)\approx\sum_i\alpha_i K(x,x_i)\)
    by enforcing the PDE
    \(\nabla\xi\cdot f=\lambda \xi\) (and normalization) in least‐squares.
\end{enumerate}

\paragraph{Numerical Experiment (The Duffing oscillator)}
We consider the Duffing oscillator
\[
\dot{x}_1 = x_2, 
\quad 
\dot{x}_2 = -\delta x_2 - x_1(\beta + \alpha x_1^2),
\]
with parameters \(\delta = 0.5\), \(\beta = -1.0\), and \(\alpha = 1.0\).  
Linearizing at the origin gives the Jacobian
\[
A = \begin{bmatrix} 0 & 1 \\ -\beta & -\delta \end{bmatrix},
\]
with eigenvalues \(\lambda_+>0\) and \(\lambda_-<0\).  

\noindent The next figures illustrate the learned eigenfunction on the grid:


\begin{figure}[h!]
\centering
\includegraphics[width=0.48\textwidth]{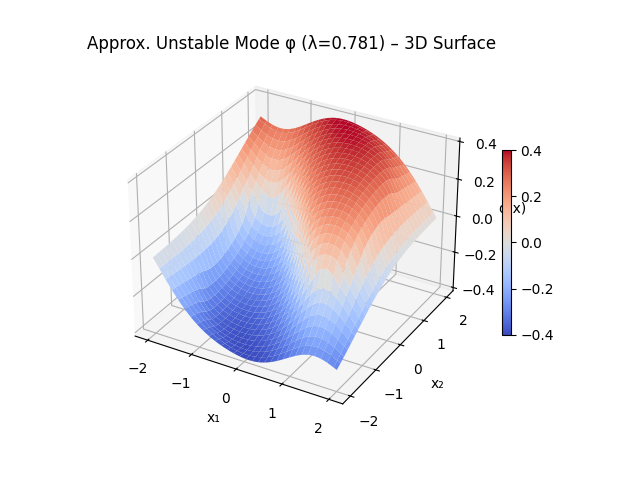}
\includegraphics[width=0.48\textwidth]{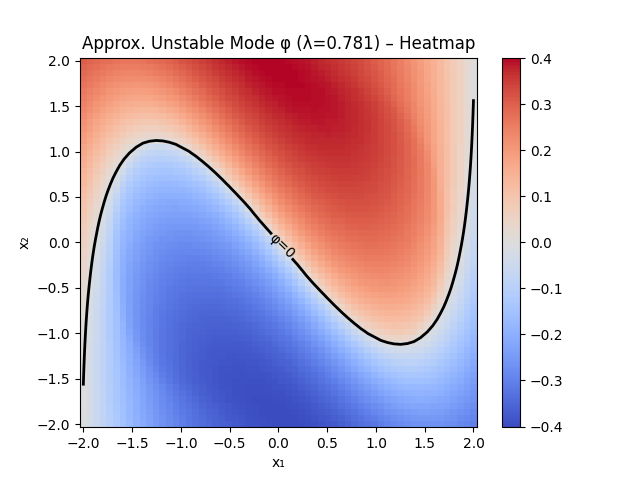}
\caption{Surface and heatmap visualizations of the Koopman eigenfunction \(\phi_{\lambda{+}}(x)\).}\label{fig_zero-level}
\end{figure}



\section{Extensions to Other Transport Equations}

In the preceding sections, we developed a unified framework for constructing reproducing kernels tailored to transport-type partial differential equations (PDEs), with a particular focus on Koopman eigenfunction approximation. Central to this framework are three complementary methods-Lions’s variational principle, Green’s function inversion, and the method of characteristics-each shown to yield equivalent kernels under suitable assumptions.

In this section, we extend these methods beyond Koopman-related PDEs to a broader class of linear transport equations commonly arising in mathematical physics and dynamical systems. These include the linear advection equation, continuity equation, Liouville equation, and transport equations with decay. While differing in their specific formulations, these PDEs all share a characteristic first-order transport structure, making them amenable to the kernel constructions previously discussed.

For each equation, we demonstrate how both the Green’s function and the resolvent kernel (characteristic flow) approaches can be systematically applied to derive reproducing kernels for their solution spaces. These examples illustrate the flexibility and extensibility of our framework, showing that the theoretical insights and computational tools developed for Koopman analysis generalize naturally to other important settings in transport theory.

\subsection{Linear Advection Equation}

The linear advection equation is one of the simplest models of transport:
\[
\frac{\partial u}{\partial t} + c\,\frac{\partial u}{\partial x} = 0,
\]
where \(c\) is a constant advection speed. For the associated eigenvalue problem, consider an operator of the form
\[
L\varphi(x)= c\,\frac{d\varphi}{dx} - \lambda\, \varphi(x).
\]

\paragraph{Green's Function Approach:}  
We seek a causal (retarded) Green's function \(G(x,\xi)\) satisfying
\[
L_x\,G(x,\xi)= \delta(x-\xi).
\]
A formal solution is given by
\[
G(x,\xi) = H(x-\xi) \exp\!\left(-\frac{\lambda}{c}(x-\xi)\right),
\]
where \(H\) is the Heaviside function enforcing causality (i.e., \(G(x,\xi)=0\) for \(x<\xi\)). The reproducing kernel is then constructed by symmetrization:
\[
K(x,y)= \int_{\mathbb{R}} G(x,\xi)\,G(y,\xi)\,w(\xi)\,d\xi,
\]
with \(w(\xi) > 0\) chosen appropriately (often as the Lebesgue measure).

\paragraph{Method of Characteristics:}  
The characteristics for the linear advection equation are straight lines given by
\[
s_t(x)= x + c\,t.
\]
Along these characteristics, a solution to the eigenvalue problem takes the form
\[
\varphi(x)= e^{\frac{\lambda}{c}(x - x_0)}\,g(x_0),
\]
where \(x_0\) is the foot of the characteristic on a prescribed initial surface. The resolvent kernel is defined by taking the Laplace transform of the flow-induced semigroup:
\[
K_\alpha(x,y)= \int_{0}^{\infty} e^{-\alpha t}\,\delta\bigl(y-s_t(x)\bigr)\,dt.
\]
An appropriate choice of \(\alpha\) and normalization then yields a reproducing kernel that is equivalent to the one obtained via the Green’s function method.

\medskip

\subsection{Continuity Equation}

The continuity equation, which expresses conservation of a physical quantity such as mass, is given by
\[
\frac{\partial \rho}{\partial t} + \nabla \cdot (\rho\, f(x)) = 0.
\]
Considering a steady velocity field \(f(x)\), one can formulate an eigenvalue problem for an operator of the form
\[
L\varphi(x)= f(x)\cdot \nabla \varphi(x) - \lambda\, \varphi(x).
\]

\paragraph{Green's Function Approach:}  
A retarded Green's function \(G(x,\xi)\) is determined by solving
\[
L_x\,G(x,\xi)= \delta(x-\xi),
\]
with suitable boundary or initial conditions imposed on a hypersurface. The reproducing kernel is then
\[
K(x,y)= \int_{\Omega} G(x,\xi)\,G(y,\xi)\,w(\xi)\,d\xi.
\]

\paragraph{Method of Characteristics:}  
The characteristics describe the trajectories along which the conserved quantity is advected. Writing a solution as
\[
\varphi(x)= \exp\!\bigl(\lambda\,T(x)\bigr)g\bigl(s_{-T(x)}(x)\bigr),
\]
where \(T(x)\) denotes the time taken for the backward characteristic to reach the initial surface, one takes the Laplace transform of the corresponding semigroup to obtain
\[
K_\alpha(x,y)= \int_{0}^{\infty} e^{-\alpha t}\, \delta\bigl(y-s_t(x)\bigr)\,dt.
\]
Under the proper choice of parameters, this resolvent kernel coincides with the one derived by the Green’s function method.

\medskip

\subsection{Liouville Equation}

In Hamiltonian dynamics, the Liouville equation governs the evolution of a phase space density:
\[
\frac{\partial \rho}{\partial t} + \{\rho, H\} = 0,
\]
where \(\{\cdot,\cdot\}\) is the Poisson bracket and \(H\) is the Hamiltonian. Rewriting this as a transport equation in phase space results in an operator of the form
\[
L\varphi(x,p)= f(x,p) \cdot \nabla_{x,p} \varphi(x,p) - \lambda\, \varphi(x,p).
\]

\paragraph{Green's Function Approach:}  
One constructs a retarded Green's function \(G((x,p),(\xi,\pi))\) by solving
\[
L_{(x,p)}\,G\bigl((x,p),(\xi,\pi)\bigr)= \delta\bigl((x,p)-(\xi,\pi)\bigr),
\]
subject to causality, and then forms the reproducing kernel
\[
K\bigl((x,p),(y,q)\bigr)= \int_{\Omega} G\bigl((x,p),(\xi,\pi)\bigr)\,G\bigl((y,q),(\xi,\pi)\bigr)\,w(\xi,\pi)\, d\xi\,d\pi.
\]

\paragraph{Method of Characteristics:}  
The phase-space trajectories serve as characteristics along which the solution is propagated:
\[
\varphi(x,p)= \exp\!\bigl(\lambda\,T(x,p)\bigr)g\bigl(s_{-T(x,p)}(x,p)\bigr).
\]
Taking the Laplace transform of the semigroup defined by these characteristics, the resolvent kernel is given by
\[
K_\alpha\bigl((x,p),(y,q)\bigr)= \int_{0}^{\infty} e^{-\alpha t}\, \delta\bigl((y,q)-s_t(x,p)\bigr)\,dt.
\]
Again, with appropriate choices, this kernel matches that obtained from the Green’s function approach.

\medskip

\subsection{Transport Equation with Decay}

A transport equation that includes a decay (or damping) term is written as
\[
\frac{\partial u}{\partial t} + f(x) \cdot \nabla u = -\lambda\, u.
\]
Here, the presence of the decay term modifies the operator to
\[
L\varphi(x)= f(x) \cdot \nabla \varphi(x) - \lambda\, \varphi(x).
\]

\paragraph{Green's Function Approach:}  
The retarded Green's function for this operator satisfies
\[
L_x\,G(x,\xi)= \delta(x-\xi),
\]
and the reproducing kernel is constructed as
\[
K(x,y)= \int_{\Omega} G(x,\xi)\,G(y,\xi)\,w(\xi)\,d\xi.
\]

\paragraph{Method of Characteristics:}  
By following the characteristic curves, a solution can be expressed as
\[
\varphi(x)= e^{\lambda T(x)}\,g\bigl(s_{-T(x)}(x)\bigr),
\]
and the Laplace transform of the associated semigroup yields the resolvent kernel
\[
K_\alpha(x,y)= \int_{0}^{\infty} e^{-\alpha t}\,\delta\bigl(y-s_t(x)\bigr)\,dt.
\]
With suitable parameter choices, this reproducing kernel is equivalent to that obtained via the Green’s function method.

\medskip

\noindent \textbf{Summary:} For each of the above transport equations-the linear advection equation, the continuity equation, the Liouville equation, and the transport equation with decay-the same two methodologies yield a reproducing kernel for the solution space. In the Green’s function approach, one inverts the operator (with appropriate causal conditions) to obtain a Green's function and then forms the kernel by symmetrization. In the method of characteristics, the solution is propagated along the flow and the Laplace transform of the resulting semigroup produces the resolvent kernel. Under proper conditions, both methods provide equivalent reproducing kernels, ensuring that the essential features and the causal dynamics of the transport process are captured uniformly.

This unified treatment demonstrates that the inversion ideas originally developed by Lions-and extended via variational and Green's function methods-are robust enough to apply to a broad class of transport equations. The resulting reproducing kernels not only permit the reconstruction of eigenfunctions via a representer theorem but also provide a powerful bridge between classical analysis and data-driven approaches for transport phenomena.

\paragraph{Conclusion.}
We have introduced a rigorous and extensible framework for constructing reproducing kernels aligned with the Koopman eigenvalue problem, based on three equivalent analytic formulations. The resulting kernels induce RKHSs whose Mercer eigenfunctions approximate true Koopman modes, and enable robust variational solvers that remain stable even in the presence of boundary singularities. Our framework is enriched with a data-adaptive kernel learning scheme and extends naturally to other first-order transport PDEs. Together, these components form a coherent, mesh-free approach to learning, representing, and computing spectral features of transport-driven systems.

\appendix

\section*{Appendix}  
\addcontentsline{toc}{section}{Appendix}

\section{Reproducing Kernel Hilbert Spaces (RKHS)} \label{sect:A}
 
We give a brief overview of reproducing kernel Hilbert spaces as used in statistical learning
theory ~\cite{CuckerandSmale}. Early work developing
the theory of RKHS was undertaken by N. Aronszajn~\cite{aronszajn50reproducing}.

\begin{definition} Let  ${\mathcal H}$  be a Hilbert space of functions on a set ${\mathcal X}$.
Denote by $\langle f, g \rangle$ the inner product on ${\mathcal H}$   and let $\|f\|= \langle f, f \rangle^{1/2}$
be the norm in ${\mathcal H}$, for $f$ and $g \in {\mathcal H}$. We say that ${\mathcal H}$ is a reproducing kernel
Hilbert space (RKHS) if there exists a function $K:{\mathcal X} \times {\mathcal X} \rightarrow \RR$
such that
\begin{itemize}
 \item[i.] $K_x:=K(x,\cdot)\in{\mathcal{H}}$ for all $x\in{\mathcal{H}}$.
\item[ii.] $K$ spans ${\mathcal H}$: ${\mathcal H}=\overline{\mbox{span}\{K_x~|~x \in {\mathcal X}\}}$.
 \item[iii.] $K$ has the {\em reproducing property}:
$\forall f \in {\mathcal H}$, $f(x)=\langle f,K_x \rangle$.
\end{itemize}
$k$ will be called a reproducing kernel  of ${\mathcal H}$. ${\mathcal H}_K$  will denote the RKHS ${\mathcal H}$
with reproducing kernel $K$ where it is convenient to explicitly note this dependence.
\end{definition}

The important properties of reproducing kernels are summarized in the following proposition.
\begin{proposition}\label{prop1} \cite{aronszajn50reproducing}
If $K$ is a reproducing kernel of a Hilbert space ${\mathcal H}$, then
\begin{itemize}
\item[i.] $K(x,y)$ is unique.
\item[ii.]  $\forall x,y \in {\mathcal X}$, $K(x,y)=K(y,x)$ (symmetry).
\item[iii.] $\sum_{i,j=1}^q\alpha_i\alpha_j K(x_i,x_j) \ge 0$ for $\alpha_i \in \RR$, $x_i \in {\mathcal X}$ and $q\in\mathbb{N}_+$
(positive definiteness).
\item[iv.] $\langle K(x,\cdot),K(y,\cdot) \rangle=K(x,y)$.
\end{itemize}
\end{proposition}
Common examples of reproducing kernels defined on a compact domain $\mathcal{X} \subset \mathrm{R}^n$ are the 
(1) constant kernel: $K(x,y)= m > 0$
(2) linear kernel: $K(x,y)=x\cdot y$
(3) polynomial kernel: $K(x,y)=(1+x\cdot y)^d$ for $d \in \NN_+$
(4) Laplace kernel: $K(x,y)=e^{-||x-y||_2/\sigma^2}$, with $\sigma >0$
(5)  Gaussian kernel: $K(x,y)=e^{-||x-y||^2_2/\sigma^2}$, with $\sigma >0$
(6) triangular kernel: $K(x,y)=\max \{0,1-\frac{||x-y||_2^2}{\sigma} \}$, with $\sigma >0$.
(7) locally periodic kernel: $K(x,y)=\sigma^2 e^{-2 \frac{ \sin^2(\pi ||x-y||_2/p)}{\ell^2}}e^{-\frac{||x-y||_2^2}{2 \ell^2}}$, with $\sigma, \ell, p >0$.

\begin{theorem} \label{thm1} \cite{aronszajn50reproducing}
Let $K:{\mathcal X} \times {\mathcal X} \rightarrow \RR$ be a symmetric and positive definite function. Then there
exists a Hilbert space of functions ${\mathcal H}$ defined on ${\mathcal X}$   admitting $K$ as a reproducing Kernel.
Conversely, let  ${\mathcal H}$ be a Hilbert space of functions $f: {\mathcal X} \rightarrow \RR$ satisfying
$\forall x \in {\mathcal X}, \exists \kappa_x>0,$ such that $|f(x)| \le \kappa_x \|f\|_{\mathcal H},
\quad \forall f \in {\mathcal H}. $
Then ${\mathcal H}$ has a reproducing kernel $K$.
\end{theorem}


\begin{theorem}\label{thm4} \cite{aronszajn50reproducing}
 Let $K(x,y)$ be a positive definite kernel on a compact domain or a manifold $X$. Then there exists a Hilbert
space $\mathcal{F}$  and a function $\Phi: X \rightarrow \mathcal{F}$ such that
$$K(x,y)= \langle \Phi(x), \Phi(y) \rangle_{\mathcal{F}} \quad \mbox{for} \quad x,y \in X.$$
 $\Phi$ is called a feature map, and $\mathcal{F}$ a feature space\footnote{The dimension of the feature space can be infinite, for example in the case of the Gaussian kernel.}.
\end{theorem}

One of the main advantages of kernel methods is that they supply \emph{analytical}
uncertainty estimates for predictions-either in the Bayesian framework  (Gaussian
processes) or via frequentist concentration bounds in the RKHS norm.

\paragraph{Bayesian predictive variance.}
Assume we observe noisy data
$y=(y_1,\dots,y_n)^{\!\top}$ with
$y_i = f(x_i) + \varepsilon_i,$ $\varepsilon_i\sim\mathcal{N}(0,\sigma^{2})$,
and place a Gaussian-process prior
$f\sim\mathcal{GP}\!\bigl(0,K(\cdot,\cdot)\bigr)$.
Let $K_{nn}=[K(x_i,x_j)]_{i,j=1}^{n}$ and
$k_x:=(K(x,x_1),\dots,K(x,x_n))^{\!\top}$.
Then the posterior at a test point $x$ is
\[
\small
\begin{aligned}
m_n(x)      &= k_x^{\!\top}\bigl(K_{nn}+\sigma^{2}I_n\bigr)^{-1}y,\\
s_n^{2}(x) &= K(x,x)-k_x^{\!\top}\bigl(K_{nn}+\sigma^{2}I_n\bigr)^{-1}k_x,
\end{aligned}
\]
where $m_n(x)$ is the predictive mean
and $s_n^{2}(x)$ is a \emph{closed-form} posterior variance
quantifying epistemic uncertainty.  
Notice that $s_n^{2}(x)\!\downarrow\!0$ as more data are taken near~$x$,
providing an explicit guide for active-learning or Bayesian-optimisation schemes.

\paragraph{Frequentist error bounds.}
For kernel‐ridge regression with tuning parameter $\lambda>0$, the estimator
$\hat f_{\,\lambda}\in{\mathcal H}_{K}$ satisfies the pointwise inequality
\begin{equation}\label{eq:pointwise-bound}
  |f(x)-\hat f_{\,\lambda}(x)|
  \;\le\;
  \bigl\|f-\hat f_{\,\lambda}\bigr\|_{{\mathcal H}_{K}}
  \sqrt{K(x,x)},
\end{equation}
so that $\sqrt{K(x,x)}$ acts as a \emph{local sensitivity coefficient}.
Coupled with standard oracle or concentration bounds on
$\|f-\hat f_{\,\lambda}\|_{{\mathcal H}_{K}}$, \eqref{eq:pointwise-bound} yields
asymptotically valid frequentist confidence regions without any resampling.

\section{Koopman Operator \& Eigenfunctions} \label{sect:B}

\subsection{Koopman Operator and its Spectrum}
In this section, we provide a brief overview of the spectral theory of the Koopman operator.
We refer the reader to \cite{mezic2020spectrum,mezic2021koopman} for more details. Consider the dynamical system
\begin{align}
    \dot x={f}(x),\label{odesys}
\end{align}
defined on a state space ${\cal Z}\subseteq \mR^n$.  The vector field, $f$, is assumed to be smooth function. Let $C^0$ denote the space of continuous functions from $\cZ$ to $\mC$ and let $\cF\subseteq C^0  $ be the function space of observable $\psi: \cZ\to \mC$.
We have following definitions for the Koopman operator and its spectrum. 
\begin{definition}[Koopman Operator] The family of Koopman  operators $\mathbb{U}_t:\cF\to \cF$ corresponding to ~\eqref{odesys} is defined as 
\begin{eqnarray}[\mathbb{U}_t \psi](x)=\psi(s_t(x)). \label{koopman_operator}
\end{eqnarray}
where $s_t(x)$ is the solution of the dynamical system (\ref{odesys}). 
If in addition $\psi$ is continuously differentiable, then $\varphi(x,t):=[\mU_t \psi ](x)$ satisfies a partial differential equation \cite{Lasota} 
\begin{align}
\frac{\partial \varphi}{\partial t}=\frac{\partial \varphi}{\partial x} f=: \cK_f \varphi \label{Koopmanpde}
\end{align}
with the initial condition $\varphi(x,0)=\psi(x)$. The operator $\cK_\bff$ is the infinitesimal generator of $\mU_t$ i.e.,
\begin{eqnarray}
{\cal K}_{f} \psi=\lim_{t\to 0}\frac{(\mathbb{U}_t-I)\psi}{t}. \label{K_generator}
\end{eqnarray}
\end{definition}
It is easy to check that each $\mU_t$ is a linear operator on the space of functions, $\cF$.  
\begin{definition}\label{definition_koopmanspectrum}[Eigenvalues and Eigenfunctions of Koopman] A function $\psi_\lambda(x)$, assumed to be at least $C^1$,  is said to be an eigenfunction of the Koopman operator associated with eigenvalue $\lambda$ if
\begin{eqnarray}
[\mU_t \psi_\lambda](x)=e^{\lambda t}\psi_\lambda(x)\label{eig_koopman}.
\end{eqnarray}
Using the Koopman generator, the (\ref{eig_koopman}) can be written as 
\begin{align}
    \frac{\partial \psi_\lambda}{\partial x}{f}=\lambda \psi_\lambda\label{eig_koopmang}.
\end{align}
\end{definition}
The {\it principal eigenfunctions} of the Koopman operator are eigenfunctions with asscoiated eigenvalues matches with the eigenvalues of the linearization of the nonlinear system at the equilibrium point. With the hyperbolicity assumption on the equilibrium point of the system (\ref{odesys}), this part of the spectrum is  well-defined. 

Equations (\ref{eig_koopman}) and (\ref{eig_koopmang}) provide a general definition of the Koopman spectrum. However, the spectrum can be defined over finite time or over a subset of the state space. The spectrum of interest to us in this paper could be well-defined over the subset of the state space. 
\begin{definition}[Open Eigenfunction \cite{mezic2020spectrum}]\label{definition_openeigenfunction}
Let $\psi_\lambda: \bC\to \mC$, where $\bC\subset \cZ$ is not an invariant set. Let $\tau^-(x)$ and $\tau^+ (x)$ be the greatest lower and least upper bounds of the time interval on which the solution $x\in  \bC$ exists, and
$\tau \in (\tau^-(x),\tau^+(x))= I_x$ be a connected open interval such that $\tau (x) \in \bC$ for all times $\tau \in I_x$.
If
\begin{align}[\mU_\tau \psi_\lambda](x) = \psi_\lambda(s_\tau(x)) =e^{\lambda \tau}  \psi_\lambda (x),\;\;\;\;\forall \tau \in I_x. 
\end{align}
Then $\psi_\lambda(x)$ is called the open eigenfunction of the Koopman operator family $\mU_t$, for $t\in \mR$ with eigenvalue $\lambda$. 
\end{definition}
If $\bC$ is a proper invariant subset of $\cZ$ in which case $I_x=\mR$ for every $x\in \bC$, then $\psi_\lambda$ is called the subdomain eigenfunction. If $\bC=\cZ$ then $\psi_\lambda$ will be the ordinary eigenfunction associated with eigenvalue $\lambda$ as defined in (\ref{eig_koopman}). The open eigenfunctions as defined above can be extended from $\bC$ to a larger reachable set when $\bC$ is open based on the construction procedure outlined in  \cite[Definition 5.2, Lemma 5.1]{mezic2020spectrum}. Let $\cP$ be that larger domain. The eigenvalues of the linearization of the system dynamics at the origin, i.e., $E:=\frac{\partial f}{\partial x}(0)$, will form the eigenvalues of the Koopman operator \cite[Proposition 5.8]{mezic2020spectrum}. Our interest will be in constructing the corresponding eigenfunctions called as principal eigenfunctions, defined over the domain $\cP$.
The principal eigenfunctions can be used as a change of coordinates in the linear representation of a nonlinear system and draw a connection to the famous Hartman-Grobman theorem  on linearization and Poincare normal form \cite{arnold2012geometrical}. 
The principal eigenfunctions will be defined over a proper subset $\cP$ of the state space $\cZ$ (called subdomain eigenfunctions) or over the entire $\cZ$ \cite[Lemma 5.1, Corollary 5.1, 5.2, and 5.8]{mezic2020spectrum}. \\

\subsection{Decomposition of Koopman Eigenfunctions}



Following \cite{deka2023path}, we decompose principal eigenfunctions into linear and nonlinear parts. Specifically, consider the decomposition of the nonlinear system into linear and nonlinear parts as
 \begin{equation}
 \dot x=f(x)=Ex+(f(x)-E x)=:E x+ f(x)\label{sys_decompose}. \end{equation}
where $E=\frac{\partial f}{\partial x}(0)$ with $Ex$ the linear part and $f(x):=f(x)-Ex$ the purely nonlinear part. 
For the simplicity of presentation and continuity of notations, we present approximation results for eigenfunctions with simple real eigenvalues; the extension to the complex case is deferred to future work. Let $\lambda$ be the eigenvalues of the Koopman generator and also of $E$. The eigenfunction corresponding to eigenvalue $\lambda$ admits the decomposition into linear and nonlinear parts.
\begin{eqnarray}
 \phi_\lambda(x)=w^\top x+h(x) \label{eigen_decompose},
\end{eqnarray}
where $ w^\top x$ and $h(x)$ are the eigenfunction's linear and purely nonlinear parts, respectively. Substituting (\ref{eigen_decompose}) in following general expression of Koopman eigenfunction i.e.,
\begin{eqnarray}
 \frac{\partial \phi_\lambda}{\partial x}(x)
 \cdot
 f(x)=\lambda \phi_{\lambda}(x)
\end{eqnarray}
and using (\ref{sys_decompose}), we obtain following equations to be satisfied by $w$ and $h(x)$
\begin{eqnarray}
w^\top E=\lambda  w^\top,\;\;\frac{\partial h}{\partial x}(x) \cdot f(x)-\lambda h(x)+w^\top f(x)=0
\label{linear_nonlinear_eig}.
 \end{eqnarray}
So, the linear part of the eigenfunction can be found as the left eigenvector with eigenvalue $\lambda$ of the matrix $E$, and the nonlinear term satisfies the linear partial differential equation.

\section{Green's Function and Characteristics} \label{sect:C}

\begin{theorem}
\label{thm:GreenFunctionSolution}
Let $f\in C^1(\overline{\Omega},\mathbb{R}^d)$ generate a complete flow $s_t$, and let $\lambda\in\mathbb{C}$ with $\mathrm{Re}\,\lambda>0$. 
Suppose $\phi:\Omega \to\mathbb{C}$ solves the Koopman eigenfunction PDE
\[
f(x)\cdot\nabla\phi(x) = \lambda \phi(x).
\]
Then
\[
\phi(x) = \int_0^\infty e^{-\lambda t} g(s_{-t}(x))\,dt,
\]
where $g$ is an auxiliary source function (e.g., boundary or initial data).
\end{theorem}

\begin{proof}
Along trajectories $t\mapsto s_t(x)$, the chain rule gives
\[
\frac{d}{dt}\phi(s_t(x)) = f(s_t(x))\cdot\nabla\phi(s_t(x)) = \lambda\phi(s_t(x)).
\]
Solving this ODE yields
\[
\phi(s_t(x)) = e^{\lambda t} \phi(x).
\]
Rearranging,
\[
\phi(x) = e^{-\lambda t} \phi(s_t(x)).
\]
Multiplying by $g(s_{-t}(x))$ and integrating over $t\in[0,\infty)$ weighted by $e^{-\lambda t}$ gives the stated formula. Convergence follows from $\mathrm{Re}\,\lambda>0$. 
\end{proof}

\begin{theorem}
\label{thm:CharacteristicsSolution}
Under the same assumptions, the solution satisfies along characteristics
\[
\phi(s_t(x)) = e^{\lambda t} \phi(x),
\quad\text{or equivalently}\quad
\phi(x) = e^{-\lambda t} \phi(s_t(x)).
\]
\end{theorem}

\begin{proof}
The method of characteristics reduces the PDE to the ODE
\[
\frac{d}{dt}\phi(s_t(x)) = \lambda\phi(s_t(x)),
\]
which integrates to
\[
\phi(s_t(x)) = e^{\lambda t}\phi(x).
\]
Rearranging yields the desired form. 
\end{proof}

\begin{theorem}[Equivalence of Green's and Characteristics Solutions]
\label{thm:GreenCharacteristicsEquivalence}
The Green's function solution and the method of characteristics solution agree.
\end{theorem}

\begin{proof}
Both constructions involve tracing the backward flow $s_{-t}(x)$ and weighting by $e^{-\lambda t}$. Integrating the characteristic solution over $t$ exactly produces the Green's function representation. 
\end{proof}

 \section{Proof of Theorem \ref{thm:unification}}
\label{sect:proof_thm_unification}

\begin{theorem}
Let $K^{\mathrm{var}}$, $K^{\mathrm{G}}$, and $K^{\mathrm{res}}_\alpha$ denote the variational, Green's function, and resolvent kernels respectively, associated with the transport operator
\[
L\phi(x) = f(x) \cdot \nabla \phi(x) - \lambda \phi(x),
\]
where $f$ is a smooth vector field generating a complete flow. Then, under suitable regularity conditions on $f$ and boundedness of $\Omega$, we have:
\[
K^{\mathrm{var}}(x, y) = K^{\mathrm{G}}(x, y) = K^{\mathrm{res}}_\alpha(x, y), \quad \text{for all } \alpha > \mathrm{Re}(\lambda).
\]
\end{theorem}

\begin{proof}
We proceed in two steps.

\paragraph{Step 1: $K^{\mathrm{var}} = K^{\mathrm{G}}$.}
Define the Hilbert space $\mathcal{H}$ with graph norm:
\[
\|u\|_{\mathcal{H}}^2 = \|u\|_{L^2(\Omega)}^2 + \|Lu\|_{L^2(\Omega)}^2.
\]
Since $f$ is smooth and $\Omega$ is bounded, the transport operator $L$ is well-defined and continuous on suitable dense subspaces like $C^\infty(\overline{\Omega})$.

Pointwise evaluation functionals may not be bounded a priori, but by construction of $\mathcal{H}$, we can define generalized evaluation via duality with test functions. Therefore, by the Riesz representation theorem (under bounded evaluation assumption or suitably mollified test functions), for each $x \in \Omega$ there exists a unique $K^{\mathrm{var}}(x, \cdot) \in \mathcal{H}$ satisfying:
\[
\langle K^{\mathrm{var}}(x, \cdot), v \rangle_{\mathcal{H}} = v(x), \quad \forall v \in \mathcal{H}.
\]

Now, define the Green's function $G(x, \xi)$ as the solution of:
\[
L_x G(x, \xi) = \delta(x-\xi)
\]
in the distributional sense. For $v \in C_c^\infty(\Omega)$:
\[
v(x) = \int_{\Omega} G(x, \xi) (Lv)(\xi) \, d\xi,
\]
by Green's formula for hyperbolic transport equations.

Define the symmetric kernel:
\[
K^{\mathrm{G}}(x, y) = \int_{\Omega} G(x, \xi) G(y, \xi) \, d\xi.
\]
Then for $v \in \mathcal{H}$:
\[
\int_{\Omega} K^{\mathrm{G}}(x, y) v(y) \, dy = \int_{\Omega} G(x, \xi) \left( \int_{\Omega} G(y, \xi) v(y) \, dy \right) d\xi.
\]
Since $G$ acts as an inverse of $L$, this reproduces $v(x)$. Thus, by uniqueness of the Riesz representer:
\[
K^{\mathrm{var}}(x, y) = K^{\mathrm{G}}(x, y).
\]

\paragraph{Step 2: $K^{\mathrm{G}} = K^{\mathrm{res}}_\alpha$.}
Using the method of characteristics, the Green’s function admits the integral representation:
\[
G(x, \xi) = \int_0^\infty e^{-\lambda t} \delta(\xi - s_t(x)) \, dt,
\]
where $s_t$ denotes the flow generated by $f$.

Thus:
\[
K^{\mathrm{G}}(x, y) = \iint_0^\infty e^{-\lambda(t_1+t_2)} \delta(s_{t_1}(x) - s_{t_2}(y)) \, dt_1 dt_2.
\]

Changing variables to $\tau = t_2 - t_1$ and simplifying (using invariance of the flow), we find:
\[
K^{\mathrm{G}}(x, y) = \int_0^\infty e^{-\lambda \tau} \left( \int_0^\infty e^{-2\lambda t} \delta(s_\tau(x) - y) \, dt \right) d\tau.
\]
The integral over $t$ gives a factor of $1/(2\lambda)$ under exponential decay, leading to:
\[
K^{\mathrm{G}}(x, y) \propto \int_0^\infty e^{-\lambda \tau} \delta(s_\tau(x) - y) \, d\tau,
\]
which matches the definition of the resolvent kernel:
\[
K^{\mathrm{res}}_\alpha(x, y) = \int_0^\infty e^{-\alpha t} \delta(y - s_t(x)) \, dt,
\]
for a suitable choice of $\alpha$ close to $\lambda$ (specifically, $\alpha > \mathrm{Re}(\lambda)$ ensures convergence).

Thus:
\[
K^{\mathrm{G}}(x, y) = K^{\mathrm{res}}_\alpha(x, y).
\]

\paragraph{Spectral and Operator Properties.}
Since $K$ is symmetric, positive semi-definite, and by construction square-integrable ($K \in L^2(\Omega \times \Omega)$ under regularity assumptions on $f$ and boundedness of $\Omega$), the integral operator $T_K$ is Hilbert–Schmidt and thus compact. If $K(x, x)$ is integrable, it is trace-class and admits a Mercer expansion:
\[
K(x, y) = \sum_{j=1}^\infty \lambda_j \phi_j(x) \phi_j(y)
\]
with convergence in the associated RKHS norm.

\paragraph{Conclusion.}
Thus, we have proven:
\[
K^{\mathrm{var}}(x, y) = K^{\mathrm{G}}(x, y) = K^{\mathrm{res}}_\alpha(x, y).
\]
\end{proof}

\bibliographystyle{IEEEtran}

\bibliography{ref_green}

\end{document}